\documentclass[12pt]{amsart}
\usepackage{amsfonts}
\usepackage{amsmath}
\usepackage{amssymb}
\usepackage{graphicx}
\usepackage{enumerate}
\usepackage{multicol}
\usepackage{mathrsfs}
\usepackage[usenames,dvipsnames]{pstricks}
\usepackage{epsfig}
\usepackage{pst-grad} % For gradients
\usepackage{pst-plot} % For axes
\usepackage[utf8]{inputenc}
\usepackage{mathdots}

\usepackage{ifthen}
\usepackage{tikz}
\usetikzlibrary{matrix,arrows,decorations.pathmorphing}
\usepackage{framed}
\input{xstring}

\newtheorem{definition}{Definition}
\newtheorem{remark}{Remark}
\newtheorem{proposition}{Proposition}
\newtheorem{theorem}{Theorem}

\newtheorem{corollary}{Corollary}
\newtheorem{lemma}{Lemma}

\newcommand{\ConC}[1]{\vspace{0.4cm}\\ (C)\hspace*{1cm}\parbox{12cm}{\em #1}\\ [0.3cm]}

\newcommand\op[1]{\mathop{\rm #1}\nolimits}

\newcommand{\n}{\mathfrak n}
\newcommand{\g}{\mathfrak g}
\newcommand{\z}{\mathfrak z}

\newcommand{\la}{\langle}
\newcommand{\ra}{\rangle}

\newcommand\R{\mathbb{R}}
\newcommand\C{\mathbb{C}}
\renewcommand\H{\mathbb{H}}
\renewcommand\O{\mathbb{O}}

\DeclareMathOperator{\End}{End}

\DeclareMathOperator{\spn}{span}

\DeclareMathOperator{\Id}{Id}
\DeclareMathOperator{\Cl}{Cl}

\newcommand\bond[1]{\draw (#1) -- +(1,0)}

\newcommand\dotbond[1]{\draw[dotted] (#1) edge +(1,0)}

\newcommand\tdots[1]{\draw (#1) ++(0.55,0) node {$\cdots$}}

\newcommand\diagbond[2]{
	\ifthenelse{\equal{#1}{u}}{
		\draw[semithick] (#2) -- +(1,0.5); }{}
	\ifthenelse{\equal{#1}{d}}{
		\draw[semithick] (#2) -- +(1,-0.5); }{} }

\newcommand\dbond[2]{
	\draw (#2) ++(0.03,0.03) -- +(0.94,0);
	\draw (#2) ++(0.03,-0.03) -- +(0.94,0);
	\ifthenelse{\equal{#1}{r}}{
		\draw[semithick] (#2) ++(0.6,0) ++(-0.15,0.2) -- ++(0.15,-0.2) -- +(-0.15,-0.2); }{}
	\ifthenelse{\equal{#1}{l}}{
		\draw[semithick] (#2) ++(0.45,0) ++(0.15,0.2) -- ++(-0.15,-0.2) -- +(0.15,-0.2); }{} }

\newcommand\tbond[2]{
	\draw (#2)  -- +(1,0);
	\draw (#2) ++(0.05,0.06) -- +(0.9,0);
	\draw (#2) ++(0.05,-0.06) -- +(0.9,0);
	\ifthenelse{\equal{#1}{r}}{
		\draw[semithick] (#2) ++(0.6,0) ++(-0.15,0.2) -- ++(0.15,-0.2) -- +(-0.15,-0.2); }{}
	\ifthenelse{\equal{#1}{l}}{
		\draw[semithick] (#2) ++(0.45,0) ++(0.15,0.2) -- ++(-0.15,-0.2) -- +(0.15,-0.2); }{} }

\newcommand\tcirc[3]{
	\ifthenelse{\equal{#1}{w}}{\filldraw[fill=white,draw=black] (#2) circle (0.08);}{}%
	\ifthenelse{\equal{#1}{b}}{\filldraw[black] (#2) circle (0.08);}{}%
	\draw (#2) node[above=2pt] {#3};
	}
\newcommand\tcross[2]{
	\draw (#1) node[above=2pt] {#2};
	\draw (#1) ++(-0.12,-0.12)-- +(0.24, 0.24);
	\draw (#1) ++(-0.12, 0.12)-- +(0.24,-0.24);
	}
\newcommand\tstar[2]{
	\draw[color=red] (#1) node {\Large$*$};
	\draw (#1) node[above=2pt] {#2};
	}
\newcommand\tsquare[2]{
		\draw[semithick,color=blue] (#1) ++(-0.15,-0.15) rectangle ++(0.3,0.3);
		\tcross{#1}{#2};
		}
\newcommand\DDnode[3]{
\ifthenelse{\equal{#1}{w}}{\tcirc{w}{#2}{#3}}{}		% white - non-compact root (Satake diagram)
\ifthenelse{\equal{#1}{b}}{\tcirc{b}{#2}{#3}}{}		% black - compact root (Satake diagram)
\ifthenelse{\equal{#1}{x}}{\tcross{#2}{#3}}{}		% crossed root (corresponding to parabolic)
\ifthenelse{\equal{#1}{s}}{\tstar{#2}{#3}}{}		% starred root (my notation for sub-parabolic)
\ifthenelse{\equal{#1}{q}}{\tsquare{#2}{#3}}{}		% crossed square (Iw root)
}

\begin{document}

\title{Rigidity of 2-step Carnot groups}

\author[M. Godoy Molina, B. Kruglikov, I. Markina, A. Vasil'ev]{Mauricio Godoy Molina, Boris Kruglikov,\\
Irina Markina, Alexander Vasil'ev}

\address{ {\it M. Godoy Molina:} Departamento de Matem\'atica y Estad\'istica, Universidad de la Frontera, Chile.\quad {\it E-mail address}: \rm{\texttt{mauricio.godoy@ufrontera.cl}}}
 %\email{mauricio.godoy@ufrontera.cl}

\address{{\it B. Kruglikov:}  Department of Mathematics and Statistics, UiT the Arctic University of Norway,
Troms\o\ 90-37, Norway.\newline
\hphantom{B} Department of Mathematics and Natural Sciences, University of Stavanger, 
40-36 Stavanger, Norway.
\quad {\it E-mail address}: \rm{\texttt{boris.kruglikov@uit.no}}}
 %\email{boris.kruglikov@uit.no}

\address{ {\it I. Markina:}  Department of Mathematics, University of Bergen, Norway.}
\email{irina.markina@uib.no}

\address{ {\it A. Vasil'ev:}
 Department of Mathematics, University of Bergen, Norway.}
\email{alexander.vasiliev@uib.no}

\thanks{The first author is partially supported by the grants Anillo ACT 1415 PIA CONICYT and DI17-0147 from Universidad de La Frontera. The second author is grateful to the Mainz Institute for Theoretical Physics
(MITP) for its hospitality and partial support during the conference GGSUSY-2017, 
where this work was reported and discussed.
The third and fourth authors are partially supported by the grants of the Norwegian Research Council \#239033/F20 and EU FP7 IRSES program STREVCOMS, grant no. PIRSES-GA-2013-612669.}

\subjclass[2010]{17B30,17B70, 16W55, 22E60}

\keywords{Clifford algebra, Clifford module, Tanaka prolongation, pseudo $H$-type algebra,
$J$-type algebra, $J^2$-condition, rigidity}

\begin{abstract}
In the present paper we study the rigidity of 2-step Carnot groups, or equivalently, of graded 2-step
nilpotent Lie algebras. We prove the alternative that depending on bi-dimensions of
the algebra, the Lie algebra structure makes it either always of infinite type or generically rigid,
and we specify the bi-dimensions for each of the choices. Explicit criteria for rigidity
of pseudo $H$- and $J$-type algebras are given. In particular, we establish the relation of
the so-called $J^2$-condition to rigidity, and we explore these conditions in relation to
pseudo $H$-type algebras.
\end{abstract}
\maketitle

%%%%%%%%%%%%%%%%%%%%%%%%%

\section{Introduction}

The study of $H$(eisenberg)-type algebras started in the 80's with A. Kaplan's seminal paper~\cite{Ka1}. These graded 2-step nilpotent Lie algebras are intricately related to representations of Clifford algebras of vector spaces endowed with a positive definite inner product, and have been extensively studied for the past 35 years, see for example~\cite{CDKR,E,E1,Ka3,KT,LT,OW}.

The extension of the work of Kaplan to Clifford algebras over non-degenerate scalar product spaces is more delicate and has only been treated in detail recently~\cite{BO,Ci,CP,FM,FMV,GKM}. In this paper, we will refer to such extensions as pseudo $H$-type algebras.

Two important algebraic results concerning the classical $H$-type algebras with a positive definite scalar product are a complete characterization of rigid $H$-type algebras, that are those with a finite Tanaka prolongation, and those $H$-type algebras satisfying a Clifford algebraic requirement known as the $J^2$-condition, see~\cite{CDKR}. Both algebraic conditions have deep implications in other aspects of the study of $H$-type groups. The rigidity of most $H$-type algebras is an obstruction to the development of an analytic deformation theory, while the presence of the $J^2$-condition has profound geometric consequences on their groups, for example, they are transitive isometry subgroups of hyperbolic spaces, they appear as the nilpotent part of Iwasawa decompositions of real rank one groups $G=KAN$, and moreover, the group $AN$ is symmetric if and only if the Lie algebra of $N$ satisfies the $J^2$-condition~\cite{CDKR}.
In this paper we discuss the analogs in the pseudo $H$-type context, and relate this to the
split versions of the division algebras.
%  No analogs are known in the pseudo $H$-type context, 
% and the aim of the present article is to initiate the corresponding study.

The fact that the Lie algebras obtained by non-degenerate indefinite bilinear forms has not been duly addressed is surprising, especially since Clifford algebras defined by non-degenerate scalar product have played a fundamental role in mathematics and physics, see~\cite{Lam,LM,Lo}. An intimately related object to the pseudo $H$-type algebras that appears naturally in mathematical physics is the notion of extended (super-)Poincar\'e algebras, see~\cite{AC,ACDV,AS1,AS2}. Some of our results have analogues in this theory. 

An important problem in sub-Riemannian geometry is to detect whether the family of automorphisms of a given non-holonomic structure on a manifold is finite dimensional~\cite{Met,OW,Pa}. The model situation is the rigidity of Carnot groups, determined by the property whether the transformation group of a left-invariant bracket generating distribution on a Lie group is a Lie group itself.
%(Carnot groups are usually considered with a positive definite or sometimes indefinite metric, but for the rigidity issue this structure is omitted, since the isometry group is always finite-dimensional).
Equivalently, linearization reduces the problem of finite-dimensionality of the automorphism group of a non-holonomic geometric structure to that for the Tanaka prolongation of an associated graded nilpotent Lie algebra. This problem is non-trivial already for 2-step nilpotent algebras, which is the main subject of our work. In the present paper we give criteria of the finite dimensionality of the Tanaka prolongation for the generalizations of $H$-type structures discussed above, and clarify the situation with the general 2-step nilpotent Lie algebras $\n=\n_{-2}\oplus\n_{-1}$ depending on their bi-dimensions $(\dim\n_{-2},\dim\n_{-1})$.

The paper is structured as follows. Section~\ref{sec:setup} is devoted to the main concepts and notations that will be in use throughout the paper. In order to keep track of our hypotheses, we have introduced the notion of $M$-type and pseudo $J$-type algebras that generalize the known objects in the positive definite context. Section~\ref{sec:rigidity} is dedicated to proving rigidity of a class of real graded 2-step nilpotent Lie algebras with the center of dimension $\geq3$ via the so-called rank one criterion that we recall and re-interpret. This class of algebras contains some important examples previously considered in the literature, see for example~\cite{OW}. This result then applies to pseudo $H$-type algebras and other cases. Section~\ref{sec:J2} deals with the $J^2$-condition of pseudo $J$-type algebras and their relation to rigidity. 

The next two sections are devoted to investigation of the generic rigidity of graded 2-step nilpotent Lie
algebras, depending on their bi-dimensions that complements the known results of P.Pansu and P.Eberlein.
We discuss the moduli spaces of the graded 2-step nilpotent Lie algebras and describe the position 
of the pseudo $H$-type algebras among the rigid ones.

We study the pseudo $H$-type algebras that satisfy the $J^2$-condition in Section~\ref{sec:J2Htype}. We prove that the classical Abelian, Heisenberg, quaternionic and octonionic $H$-type algebras, and their split analogs that we introduce exhaust all possible 
pseudo $H$-type algebras with the $J^2$-condition. 
In Appendix A we relate these algebras to division algebras and their split versions. 
In Appendix B we relate them to the nilradicals of parabolics in simple Lie algebras.

%%%%%%%%%%%%%%%%%%%%%%%%%

\section{Pseudo $J$-type and pseudo $H$-type algebras}\label{sec:setup}

Let $\n=\n_{-2}\oplus\n_{-1}$ be a real or complex graded 2-step nilpotent Lie algebra, and let $\langle\cdot,\cdot\rangle$ be a non-degenerate real or complex symmetric bilinear form. We assume that the restriction $\langle\cdot,\cdot\rangle_{\n_{-2}}$ of $\langle\cdot,\cdot\rangle$ to the subspace $\n_{-2}$ is also non-degenerate and the decomposition $\n=\n_{-2}\oplus\n_{-1}$ is orthogonal (such a choice can be made on any 2-step algebra $\n$).
We call the pair $(\n,\langle\cdot,\cdot\rangle)$ an {\it $M$-type Lie algebra}, since such objects have been referred to as
``metric Lie algebras'' in the real case with a positive definite symmetric bilinear form, see, for example~\cite{E1}. A complex Lie algebra may have different real forms. Each real form carries a real non-degenerate symmetric bilinear form, whose complexification coincides with the original complex non-degenerate symmetric bilinear form.

\begin{definition}\label{def:Jrc}
Let $(\n,\langle\cdot,\cdot\rangle)$ be an $M$-type Lie algebra. The linear representation
$J\colon\n_{-2}\to \End(\n_{-1})$ defined by
\begin{equation}\label{def:J}
\la J_zx,y\ra_{\n_{-1}}=\la z,[x,y]\ra_{\n_{-2}}\quad\text{for all}\quad x,y\in\n_{-1},\ \ z\in\n_{-2}.
\end{equation}
is called the $J$-map of $\n$.
\end{definition}

\begin{definition}\label{def:Jtype}
A real $M$-type algebra $(\n,\langle\cdot,\cdot\rangle)$ is of pseudo $J$-type if there is an orthonormal basis $\{z_1,\ldots,z_m\}$ for $\n_{-2}$ such that the $J$-maps satisfy the condition
\begin{equation}\label{eq:J2Cliff}
J_{z_i}^2=\pm\Id_{\n_{-1}},\qquad i=1,\ldots,m.
\end{equation}
\end{definition}

Note that this definition extends the notion of a $J$-type algebra in~\cite[Section~6]{OW}. It is relevant to note that we consider the operators $J_{z_i}$  as complex structures for some indices $i=1,\ldots,m$, but we also allow involutions for the rest of them. Moreover, the identities~\eqref{eq:J2Cliff} are not necessarily related to the Clifford condition $J_{z_i}^2=-\langle z_i,z_i\rangle_{\n_{-2}}{\rm Id}_{\n_{-1}}$ defined on the chosen basis for $\n_{-2}$, as it has been used in~\cite{OW} in the presence of a positive definite scalar product. A particular case of these pseudo $J$-type algebras are the pseudo $H$-type algebras.

\begin{definition}\label{def:Htype}
A real $M$-type algebra $(\n,\langle\cdot,\cdot\rangle)$ is said to be of pseudo $H$-type if the $J$-maps satisfy the Clifford relations
\begin{equation}\label{eq:anticom}
J_{z_i}J_{z_j}+J_{z_j}J_{z_i}=-2\langle z_i,z_j\rangle_{\n_{-2}}{\rm Id}_{\n_{-1}},\quad i,j=1,\ldots,m,
\end{equation}
for a basis $\{z_1,\ldots,z_m\}$ of $\n_{-2}$ (that can be chosen orthonormal).
\end{definition}

Note that equation~\eqref{eq:anticom} implies that the $J$-map can be extended to a representation of the Clifford algebra ${\rm Cl}(\n_{-2}, \langle\cdot,\cdot\rangle_{\n_{-2}})$ on the space $\n_{-1}$. Condition~\eqref{eq:anticom} implies condition~\eqref{eq:J2Cliff} for a special choice of signs, but not conversely. The classical $J$-type and $H$-type algebras originated from the papers~\cite{Ka1,Met} 
in which $H$-type algebras were defined as generalizations of the Heisenberg algebra endowed 
with a positive definite scalar product. Their pseudo-analogs were introduced in~\cite{Ci,GKM}.

The construction of pseudo $H$-type Lie algebras is delicate, so we postpone its precise description to Section~\ref{sec:J2Htype}. Nevertheless, let us describe them briefly here. We denote by $\mathbb R^{r,s}$ the vector space $\mathbb R^{r+s}$ equipped with the metric
 \[
\langle v,w \rangle_{r,s}=\sum_{i=1}^{r}v_iw_i-\sum_{j=r+1}^{r+s}v_jw_j,\quad v=(v_i),w=(w_i)\in \mathbb R^{r+s}.
 \]
The pseudo $H$-type Lie algebra $\n^{r,s}(V)$ is a real graded 2-step nilpotent Lie algebra structure on the space ${\mathbb R}^{r,s}\oplus V$, where $V$ is an admissible ${\rm Cl}({\mathbb R}^{r,s})$-module
(more details in Section \ref{sec:J2Htype}). If the Clifford module $V$ is irrelevant to the statement we will denote by $\n^{r,s}$ the class of all pseudo $H$-type algebras with the same center. 
Equivalent definitions of pseudo $H$-type algebras can be found in~\cite{Ci,GKM,Ka1}.

The $H$-type algebras are a special case of a wider class of algebras satisfying the hypothesis ($H$\,),
introduced in~\cite{Met}, which states that
$\omega_\alpha(x,y)=\alpha([x,y])$ is a non-degenerate 2-form on $\n_{-1}$ for any
non-zero $\alpha\in\n_{-2}^*$. % ($x,y\in\n_{-1}$)
In~\cite{E,E1,LT,MS} the authors established the following equivalent definitions.

\begin{proposition}\label{prop:Met}
The following statements are equivalent for a real graded 2-step nilpotent Lie algebra $\n=\n_{-2}\oplus\n_{-1}$:
\begin{enumerate}
\item $\n$ satisfies M\'etivier's hypothesis (H).
\item ${\rm ad}_x\colon\n_{-1}\to\n_{-2}$ is surjective for any non-zero $x\in\n_{-1}$.
\item $J_z\colon\n_{-1}\to\n_{-1}$ is a non-degenerate map for any non-zero $z\in\n_{-2}$.
\end{enumerate}
\end{proposition}

%%%%%%%%%%%%%%%%%%%%%%%%%

\section{Rigidity of $M$-type algebras}\label{sec:rigidity}

Given a graded nilpotent Lie algebra $\n=\bigoplus_{i=1}^s\n_{-i}$ generated by $\n_{-1}$, there is an algebraic procedure to compute symmetries of $\n_{-1}$, called the Tanaka prolongation~\cite{T}.
This is the maximal graded Lie algebra $\hat\n=\n_{-s}\oplus\dots\oplus\n_{-1}\oplus\n_0\oplus\dots$ with
$\hat\n_{<0}=\n$. There is a vast literature on this topic, and we refer the reader to~\cite{AK,CS,M93,OW,Z}
and the references therein for an overview.

A graded 2-step nilpotent Lie algebra ${\mathfrak n} = {\mathfrak n}_{-2} \oplus{\mathfrak n}_{-1}$ is called {\it rigid} or of {\it finite type} if its {\it Tanaka prolongation} $\hat{\mathfrak n} = {\mathfrak n}_{-2} \oplus {\mathfrak n}_{-1} \oplus {\mathfrak n}_0\oplus\cdots$ is finite dimensional. Otherwise it is called of {\it infinite type}. These definitions are valid for either complex or real Lie algebras. We will assume throughout the text that $\mathfrak{n}$ is {\it fundamental} or equivalently {\it stratified}, i.e.\ $\mathfrak{n}_{-1}$ generates
$\mathfrak{n}$, and  in the 2-step case that $\mathfrak{n}_{-1}$ contains no central elements, so
 the center $\z$ of $\n$ is exactly $\n_{-2}$, see \cite{T}.

A criterion to detect whether a given {\em complex} graded 2-step nilpotent Lie algebra ${\mathfrak n} = {\mathfrak n}_{-2} \oplus{\mathfrak n}_{-1}$ is rigid is the {\it corank one criterion} (see~\cite[Theorem 1]{DR} and~\cite{OW2,K}, which are based on~\cite[Corollary 2 of Theorem 11.1]{T}). It states that ${\mathfrak n}$ is of infinite type if and only if there exist $x\in{\mathfrak n}_{-1}$ and a hyperplane $\Pi\subset{\mathfrak n}_{-1}$, such that
\[
[x,y]=0\quad\text{for all}\quad y\in\Pi.
\]
A key observation is that the corank one criterion can be rewritten conveniently in the case when the adjoint map of $\n$ induces endomorphisms of $\n_{-1}$ through a non-degenerate symmetric bilinear form, as in Definition~\ref{def:J}. We conclude that the above is equivalent to the existence of a non-zero vector $x \in{\mathfrak n}_{-1}$ with
\[
J_{{\mathfrak n}_{-2}}x\in\Pi^\bot,\quad\text{or}\quad \dim_{\mathbb C} J_{{\mathfrak n}_{-2}}x=\dim(\Pi^\bot)=1.
\]
This means that $J_{{\mathfrak n}_{-2}}$ is a one dimensional complex subspace of $\End(\n_{-1})$
and therefore there exists
 $$
L\subset{\mathfrak n}_{-2},\quad {\rm codim_{\mathbb C}}\; L=1,\quad\text{such that}\quad J_Lx=0.
 $$
The last condition is equivalent to the property $x\in\bigcap_{z\in L}\ker(J_z)$, so we conclude:
 \begin{gather}
\n\textit{ is of infinite type if and only if there exists a subspace }\notag\\
L\subset{\mathfrak n}_{-2},\quad {\rm codim}_{\mathbb C}\;L=1,\quad\text{such that}\quad\bigcap_{z\in L}\ker(J_z)\not=\{0\}.
\label{Lcriterion}
 \end{gather}

Notice that for real graded 2-step nilpotent Lie algebras $\n$ the criterion fails in its sufficient part
and the complexification is required, i.e. condition~\eqref{Lcriterion} will be applied to
the complexifications $\n^\mathbb{C}$ and $J^\mathbb{C}$.

 \begin{remark}\label{RK1}
{\rm Even though this rigidity criterion is formulated in terms of $M$-type algebras, viz.\
the definition of the operators $J_z$ involves a choice of a scalar product on $\n$, the output, 
namely the alternative whether the algebra $\n$ is of finite or infinite type, 
does not depend on this choice.}
 \end{remark}

Since the Tanaka prolongation is a linear algebra operation, the following folklore-known statement relates the complexification and the prolongation of real Lie algebras.

\begin{proposition}\label{prop:complex}
Let ${\mathfrak n}$ be a real graded nilpotent Lie algebra of any step, and let ${\mathfrak n}^{\mathbb C}$ be its complexification. Then $\hat{\mathfrak n}^{\mathbb C}=\widehat{{\mathfrak n}^{\mathbb C}}$.
 \end{proposition}

This statement has the following immediate implications.

\begin{corollary}\label{coro:complex}
Let ${\mathfrak n}$ and $\tilde{\mathfrak n}$ be two graded nilpotent Lie algebras such that ${\mathfrak n}^{\mathbb C}\cong\tilde{\mathfrak n}^{\mathbb C}$. Then they are either simultaneously rigid or of infinite type.
\end{corollary}

\begin{corollary}\label{coro:realcomp}
Let ${\mathfrak n}$ be a real graded nilpotent Lie algebra of any step with the complexification ${\mathfrak n}^{\mathbb C}$. Then $\n$ and $\n^{\mathbb C}$ are either simultaneously rigid or of infinite type in their respective categories.
\end{corollary}

\begin{proof}
The claims follow from the equality $\dim_{\mathbb R}\n=\dim_{\mathbb C}\n^{\mathbb C}$.
\end{proof}

The idea of the following constructions is to employ the complexification and criterion \eqref{Lcriterion}
in order to detect the rigidity of real 2-step nilpotent Lie algebras. The following is a special
case of Theorem 4 in~\cite{DR}, see also~\cite[Lemma 6]{OW}.

\begin{lemma}\label{lem:inftypeC}
Let ${\mathfrak n}={\mathfrak n}_{-1}\oplus{\mathfrak n}_{-2}$ be a real $M$-type algebra.
If $\dim\n_{-2} = 2$, then $\n$ is of infinite type.
\end{lemma}

\begin{proof}
Let us complexify $\n$ and use Corollary \ref{coro:realcomp}.
Choose a basis $z_1,z_2$ of ${\mathfrak n}^{\mathbb C}_{-2}$. For some $\lambda\in{\mathbb C}$
the vector $z=z_1+\lambda z_2$ is null and so $\det J_{z_1+\lambda z_2}=0$. Then one can take
$L=\spn\{z\}$ and the claim follows from $\ker(J_z)\neq0$.
\end{proof}

The aim of this section is to generalize~\cite[Theorem 1]{OW} to a much wider class of real Lie algebras. In order to state the result, we introduce the following condition (C): an $M$-type algebra $\n=\n_{-2}\oplus\n_{-1}$ satisfies this condition if
%  \begin{gather*}\label{ConC}
% \textit{There exist three linearly independent vectors } z_1,z_2,z_3\in\n_{-2}\\
% \textit{such that the } J\textit{-maps } J_{z_i} \textit{ are non-degenerate, and }
% J_{z_i}J_{z_j}=\sigma_{ij}J_{z_j}J_{z_i}, \tag{C}\\
% \textit{ for all } i,j\in\{1,2,3\}, \textit{ where } \sigma_{ij}\in\{-1,1\} \textit{ and }
% \sigma_{12}\sigma_{13}\sigma_{23}=-1.
%  \end{gather*}
 %
\ConC{There exist three linearly independent vectors $z_1,z_2,z_3\in\n_{-2}$ such that the $J$-maps
$J_{z_i}$ are non-degenerate, and $J_{z_i}J_{z_j}=\sigma_{ij}J_{z_j}J_{z_i}$, for all $i,j\in\{1,2,3\}$,
where $\sigma_{ij}\in\{-1,1\} $ and $\sigma_{12}\sigma_{13}\sigma_{23}=-1$.}
 %\vspace{0.2cm}

%{\it There exist three linearly independent vectors $z_1,z_2,z_3\in\n_{-2}$ such that the $J$-maps $J_{z_i}$ are non-degenerate, and
%$
%J_{z_i}J_{z_j}=\sigma_{ij}J_{z_j}J_{z_i},
%$
%for all $i,j\in\{1,2,3\}$, where $\sigma_{ij}\in\{-1,1\} $ and $
%\sigma_{12}\sigma_{13}\sigma_{23}=-1$.
%}
%\\

Condition (C) is not too restrictive, and several important systems satisfy it. For instance, it holds
for the pseudo $H$-type algebras as well as in the case when $J_{z_i}$ commutes with both $J_{z_j}$ and $J_{z_k}$, but $J_{z_j}$ and $J_{z_k}$ anti-commute.

 \begin{theorem}\label{th:JJJcomp}
Let $\n=\n_{-2}\oplus\n_{-1}$ be a real $M$-type algebra with $\dim{\mathfrak n}_{-2} \geq 3$ that satisfies condition {\rm (C)}. Then ${\mathfrak n}$ is rigid.
 \end{theorem}

 \begin{proof}
Let $z_1,z_2,z_3\in\n_{-2}$ be the linearly independent vectors from condition (C) and with a slight abuse of notation define $K={\rm span}_{\mathbb C}\{z_1,z_2,z_3\}\subset\n_{-2}^{\mathbb C}$.

Suppose, that $\n^{\mathbb C}$ is of infinite type and let $L\subset\n_{-2}^{\mathbb C}$ be the codimension one subspace coming from the corank one criterion. By the dimension count, it is easy to see that $\dim_{\mathbb C}(K\cap L)\geq 2$.

Without loss of generality, we can assume that there are $s_1, s_2 \in{\mathbb C}$ such that
$K\cap L\supset\spn_{\mathbb C}\{z_1-s_1z_3, z_2-s_2z_3\}$.
By the definition of $L$, there is a non-zero $x \in{\mathfrak n}_{-1}^{\mathbb C}$ such that
 \[
% (J_{z_1}^{\mathbb C} - s_1J_{z_3}^{\mathbb C})(x) = (J_{z_2}^{\mathbb C} - s_2J_{z_3}^{\mathbb C})(x) = 0.
(J_{z_1}- s_1J_{z_3})(x) = (J_{z_2}- s_2J_{z_3})(x) = 0,
 \]
where we use the notation $J_z$ for the complexification as well. By the non-degeneracy of $J_{z_k}$, $k\in\{1,2,3\}$, it holds that $s_1,s_2\neq0$. By condition (C)
 \[
J_{z_1}J_{z_2}x = s_2J_{z_1}J_{z_3}x=\sigma_{13}s_2J_{z_3}J_{z_1}x=\sigma_{13}s_1s_2J_{z_3}^2x,
 \]
and analogously, $J_{z_2}J_{z_1}x =  \sigma_{23} s_1 s_2J_{z_3}^2 x$.
Thus, it follows that $\sigma_{23} = \sigma_{12} \sigma_{13}$. This contradicts $\sigma_{12}\sigma_{13}\sigma_{23} = -1$ in condition (C), and so ${\mathfrak n}^{\mathbb C}$ is rigid. The rigidity of $\n$ follows from Corollary~\ref{coro:realcomp}.
 \end{proof}

Observe that Theorem~\ref{th:JJJcomp} is applicable in a broader context than just for $M$-type algebras. In particular, it holds when the symmetric bilinear form $\langle\cdot,\cdot\rangle$ degenerates on $\n_{-2}$, as long as the latter possesses a three dimensional subspace satisfying condition (C). Such degenerate cases have been considered before in the literature, see~\cite{BO,CW,CP}.

As a consequence of Theorem~\ref{th:JJJcomp} we have the following results, see also~\cite{Reim}.

\begin{corollary}
Any pseudo $H$-type algebra with $\dim{\mathfrak n}_{-2} \geq 3$ is rigid.
\end{corollary}

\begin{corollary}\label{coro:comeig}
Let $\n$ be a pseudo $J$-type algebra, and let $A$ be the subalgebra of ${\rm End}(\n_{-1})$ generated by the set $\{J_zJ_w\colon z,w\in\n_{-2}\}$. If $\n$ is of infinite type, then $A$ has a common eigenvector over ${\mathbb C}$.
\end{corollary}
\begin{proof}
Choose an orthonormal basis $\{z_1,\dotsc,z_m\}$ of $\n_{-2}$. As in the proof of Theorem~\ref{th:JJJcomp}, by complexifying, we know there exist non-zero complex numbers $s_1,\dotsc,s_{m-1}\in{\mathbb C}$ such that (perhaps after re-enumeration)
\begin{equation}\label{eq:JinL}
J_{z_i}-s_iJ_{z_m}\in J_L,
\end{equation}
where $L$ is the subspace given in the corank one condition. If $x\in\n_{-1}^{\mathbb C}$ is the non-zero vector corresponding to $L$, then equation~\eqref{eq:JinL} implies that there are constants $c_{ij}\in{\mathbb C}$, $i,j\in\{1,\dotsc,m\}$, such that
$J_{z_i}J_{z_j}x=c_{ij}x$.
Consequently, since any $I\in A$ can be written as a polynomial in the variables $J_{z_i}J_{z_j}$, say $I=P(\dotsc,J_{z_i}J_{z_j},\dotsc)$, we conclude that
$Ix=P(\dotsc,J_{z_i}J_{z_j},\dotsc)x=P(\dotsc,c_{ij},\dotsc)x.$ \qedhere
\end{proof}

\begin{remark}
{\rm Let $z_1,z_2,z_3\in\n_{-2}$ be linearly independent.
Corollary~\ref{coro:comeig} implies that if $[J_{z_i}, J_{z_j}]$, $1\leq i<j\leq3$,
have no common eigenvector over ${\mathbb C}$, then ${\mathfrak n}$ is rigid.}
\end{remark}

%%%%%%%%%%%%%%%%%%%%%%

\section{Rigidity and $J^2$-condition}\label{sec:J2}

The main goal of this section is to characterize an analogue of the so-called $J^2$-condition for pseudo $J$- and $H$-type algebras, studied in~\cite{Ci2,CDKR}. Although in the classical situation this condition has deep geometric implications, we only focus here on those  algebras that admit this algebraic property.

\begin{definition}\label{def:J2}
A pseudo $J$-type algebra $\n=\n_{-2}\oplus\n_{-1}$ satisfies the $J^2$-condition if for every
$x\in\n_{-1}$, $\langle x,x\rangle_{\n_{-1}}\neq 0$, and for every orthogonal pair $z,z'\in\n_{-2}$, there exists $z''\in{\n_{-2}}$ satisfying
\begin{equation}\label{eq:J2}
J_zJ_{z'}x=J_{z''}x.
\end{equation}
\end{definition}

Equation~\eqref{eq:J2} implies that for any given $x\in\n_{-1}$, $\langle x,x\rangle_{\n_{-1}}\neq 0$, 
the space
 \begin{equation}\label{Ax}
A_x=\R x\oplus J_{\n_{-2}}x=\{\alpha x+J_{z'}x\mid \alpha\in\R, z'\in\n_{-2}\}
 \end{equation}
is $J_z$-invariant for every $z\in\n_{-2}$ (note that
$\R x\cap J_{\n_{-2}}x=0$ if $\|x\|^2_{\n_{-1}}\neq0$). 
 % Moreover it will be given an algebra structure in Section ...

The converse statement is true for pseudo $H$-type algebras:
if $J_zA_x\subset A_x$ for any non-null $x\in\n_{-1}$ and any $z\in{\n_{-2}}$,
then the pseudo $H$-type algebra $\n$ satisfies the $J^2$-condition.
This implication holds thanks to the identity
 \begin{equation}\label{JJxx}
\langle J_{z}x,J_{z'}x\rangle_{\mathfrak{n}_{-1}}=
\langle z,z'\rangle_{\mathfrak{n}_{-2}}\langle x,x\rangle_{\mathfrak{n}_{-1}}
 \end{equation}
that will be used in the sequel.
 % Nevertheless, the pseudo $J$-type algebras, which are not of pseudo $H$-type do not possess this property. 
 
If $\langle x,x\rangle_{\n_{-1}}\neq 0$, then $A_x$ is a non-degenerate subspace of
$(\n_{-1},\langle\cdot,\cdot\rangle_{n_{-1}})$.
 % and the $J^2$-condition applied to pseudo $H$-type algebras is equivalent to the existence of an admissible Clifford module of a special type.
In the case when $x$ is null, the restriction of $\langle\cdot,\cdot\rangle_{n_{-1}}$ to $A_x$ 
is degenerate. We define the general $J^2$-condition by omitting the requirement $\langle x,x\rangle_{\n_{-1}}\neq 0$.

\begin{definition}\label{def:J2general}
A pseudo $J$-type algebra $\n=\n_{-2}\oplus\n_{-1}$ satisfies the general $J^2$-condition if for any
$x\in\n_{-1}$, and for all orthogonal pairs $z,z'\in\n_{-2}$ there exists $z''\in{\n_{-2}}$ such that  equation~\eqref{eq:J2} holds.
\end{definition}

Observe that for $J$-type algebras (more generally, when $\langle\cdot,\cdot\rangle$
is positive definite on $\n_{-1}$) the general $J^2$ condition is equivalent to the usual $J^2$
condition. Of course, the general $J^2$ condition implies the usual one.
Note also that the (general) $J^2$-condition is trivially fulfilled for $\dim(\n_{-2})=0,1$.

% \begin{lemma}\label{lem:dimnm2e2}
% No pseudo $J$-type algebra $\n=\n_{-2}\oplus\n_{-1}$ with $\dim\n_{-2}=2$ satisfies the (general) $J^2$-condition.
% \end{lemma}

 \begin{lemma}\label{lem:dimnm2e2}
The $J^2$-condition is never satisfied for $\dim\n_{-2}=2$.
 \end{lemma}

 \begin{proof}
Assuming the opposite we get two linear operators $J_i=J_{z_i}\in\op{GL}(\n_{-1})$ with $J_i^2=\epsilon_i=\pm1$, $i=1,2$,
for an orthonormal basis $z_1,z_2\in\n_{-2}$.

Note that if $z=s_1z_1+s_2z_2\neq0$, $s_i\in\mathbb{R}$, then $J_z\neq\lambda_z\cdot{\bf1}$
because $\lambda_z=0$ contradicts non-degeneracy of
$\langle\cdot,\cdot\rangle_{n_{-2}}$, while for $\lambda_z\neq0$ the operator $J_z$ is not skew-symmetric with respect to $\langle\cdot,\cdot\rangle_{n_{-1}}$.

Now the $J^2$-condition implies that for any non-null $x\in\n_{-1}$ for some scalar functions
$\alpha_x,\beta_x$ we have the equality: $J_1J_2x=(\alpha_xJ_1+\beta_xJ_2)x$.
If $\alpha_x\equiv0$ then $J_1\equiv\beta\cdot{\bf1}$, and if $\beta_x\equiv0$ then $J_1\equiv\alpha\cdot{\bf1}$.
Consequently, $\alpha_x\beta_x\not\equiv0$. Thus  application of the operator $J_1$ to the above equality from the left
 %and exploiting it again
yields
 \begin{equation}\label{J1J2}
J_1x=(p_x{\bf1}+q_xJ_2)x
 \end{equation}
for some (perhaps rational) scalar functions $p_x,q_x$. Since for generic $x\in\n_{-1}$
the vectors $x$ and $J_2x$ are linearly independent, the linearity of the operators $J_1,{\bf1},J_2$
implies that $p_x=p=\op{const}$, $q_x=q=\op{const}$ in formula \eqref{J1J2}, i.e.\
$J_1=p\,{\bf1}+qJ_2$.

Squaring this latter identity yields $\epsilon_1{\bf1}=(p^2+\epsilon_2q^2){\bf1}+2pqJ_2$.
Then $pq=0$, but since $q\neq0$ we get $p=0$, $q=\pm1$, whence $J_2=\pm J_1$ and $J_z=0$ for $z=z_1\mp z_2$,
which is a contradiction.
 \end{proof}

 \begin{theorem}\label{th:Boris}
A pseudo $J$-type algebra $\n=\n_{-2}\oplus\n_{-1}$, $\dim(\n_{-2})\geq 3$, satisfying
the general $J^2$-condition is rigid.
 \end{theorem}

\begin{proof} We will work over the field $\mathbb C$. If necessary we use the complexification and finish the proof by applying Corollary~\ref{coro:realcomp}.

By contradiction, let us suppose that a pseudo $J$-type algebra $\n$ is of infinite type. Corollary~\ref{coro:comeig} implies that there exists a common eigenvector
$x\in\n_{-1}$ for the algebra $A$ generated by the operators $J_{z_i}J_{z_j}$, where $\{z_i\}_{i=1}^m$ is an orthonormal basis for $\n_{-2}$. Thus, we obtain
$J_{z_i}J_{z_j}x=\mu_{z_i,z_j}x$ for some non-vanishing $\mu_{z_i,z_j}\in\mathbb{C}$.

We claim that the same $x$ is also an eigenvector for all $J_z$, $z\in\n_{-2}$. Indeed, since the pseudo $J$-type algebra satisfies Definition~\ref{def:J2general}, for any pair $z_i, z_j$ of vectors from the orthonormal basis for $\n_{-2}$, there is $z''_{ij}\in\n_{-2}$ such that
$J_{z_i}J_{z_j}x=J_{z''_{ij}}x$. We get $J_{z_j}x=\epsilon_iJ_{z_i}J_{z''_{ij}}x=
\pm\mu_{z_i,z''_{ij}}x$ by $J_{z_i}^2=\epsilon_i=\pm{\rm Id}_{n_{-1}}$. Since $j$ is an arbitrary index and $\{z_j\}_{i=1}^m$ is a basis for $\n_{-2}$, the claim follows.

Let us fix the eigenvector $x$. Then  $J_z x=\lambda_z x$ for a non-zero linear function $\lambda\colon \n_{-2}\to\mathbb{C}$.
The definition of pseudo  $J$-type algebras leads to
 \begin{equation}\label{Jj}
\langle z,[x,y]\rangle_{\n_{-2}}=\langle J_zx,y\rangle_{\n_{-1}}=\lambda_z\langle x,y\rangle_{\n_{-1}}
=-\langle x,J_zy\rangle_{\n_{-1}}
 \end{equation}
 for all $y\in\n_{-1}$, $z\in\n_{-2}$.
This implies that $\langle x,x\rangle_{n_{-1}}=0$, and thus $[x,\Pi]=0$ for the co-dimension 1 subspace $\Pi\subset\n_{-1}$ orthogonal to $x$. We can choose $y\not\in\Pi$ such that $\langle x,y\rangle_{\n_{-1}}=1$.
Thus, there exists a basis $e_1=y,e_2=x,e_3,\dots,e_n$, $n=\dim\n_{-1}$, $e_i\in\Pi$ for $i>1$,
such that $\langle e_1,e_2\rangle_{\n_{-1}}=1$, $\langle e_2,e_i\rangle_{\n_{-1}}=0$ for $i>1$.

From~\eqref{Jj} with $x=e_2$ and $y=e_1$ we obtain $\lambda_z=-\langle e_2,J_ze_1\rangle_{\n_{-1}}$,
whence $J_ze_1=-\lambda_ze_1\!\!\mod\Pi$.
Similarly, substituting $y=e_i$ for $i>2$ in \eqref{Jj} we obtain $\langle e_2,J_ze_i\rangle_{\n_{-1}}=0$,
whence $J_ze_i=0\!\!\mod\Pi$ for $i>2$.
Finally, since $e_2=x$ is an eigenvector, we obtain $J_ze_2=\lambda_z e_2$.
Hence, the matrix of the operator $J_z\in\op{End}(\n_{-1})$ in the chosen basis $e_1,e_2,\dots,e_n$
has the form
 \begin{equation}\label{mat1}
J_z=\begin{pmatrix}
-\lambda_z & 0 & 0 & \dots & 0\\
* & \lambda_z & * & \dots & *\\
* & 0 & * & \dots & *\\
* & \vdots & \ddots & \vdots & *\\
* & 0 & * & \dots & *
\end{pmatrix}.
\end{equation}
Taking $z'$ orthogonal to $z$ and multiplying~\eqref{mat1} by $J_{z'}$ from the right, we arrive at
 $$
J_zJ_{z'}=\begin{pmatrix}
\lambda_z\lambda_{z'} & 0 & 0 & \dots & 0\\
* & \lambda_z\lambda_{z'} & * & \dots & *\\
* & 0 & * & \dots & *\\
* & \vdots & \ddots & \vdots & *\\
* & 0 & * & \dots & *
\end{pmatrix}
 $$
that differs from any matrix $J_{z''}$ unless $\lambda_z\lambda_{z'}=\lambda_{z''}=0$.
Since for a generic orthogonal pair $z,z'\in\n_{-2}$, it holds that $\lambda_z\neq 0$, $\lambda_{z'}\neq0$, we obtain a contradiction.
This proves the required rigidity of $\n$.
\end{proof}

%%%%%%%%%%%%%%%%%%%%%%%%%

\section{Generic rigidity of 2-step nilpotent algebras}\label{sec:add}

Let us discuss the general rigidity problem for graded 2-step nilpotent algebras $\n=\n_{-2}\oplus\n_{-1}$.
%(which we continue to assume fundamental, otherwise rigidity fails).
Notice that if $0\le \dim\n_{-2}\le2$, then the algebra $\n$ is of infinite type, see Section~\ref{sec:rigidity}. So we study the algebras with $\dim\n_{-2}>2$.

In this section we work with both real and complexified cases, so we omit specification of the field
and signature for the metric $\langle\cdot,\cdot\rangle_{\n_{-1}}$ and simply write $\mathfrak{so}(n)$ below (the conclusion does not depend on this signature).

Denote by $N(m,n)$ the space of 2-step graded nilpotent
Lie algebras with bi-dimensions $(m,n)$, where $m=\dim\n_{-2}$, $n=\dim\n_{-1}$, $0\le m\le\binom{n}2$.
This space is an algebraic manifold of dimension $m d$, $d=\binom{n}2-m$, with the isomorphism to
the Grassmanian $\op{Gr}_d(\Lambda^2\n_{-1})$ given by associating
$Z=\ker(\Lambda^2\n_{-1}\to\n_{-2})$ to the bracket on $\n$. Reciprocally, $\n$ is
restored by letting $\n_{-2}=\Lambda^2\n_{-1}/Z$. % and taking the natural Lie bracket.
In particular, the notion of a generic
Lie algebra structure $\n\in N(m,n)$ is given by the notion of a Zariski generic point
$Z\in\op{Gr}_d(\Lambda^2\n_{-1})$.

 \begin{theorem}\label{th:3}
A generic algebra $\n\in N(m,n)$ is rigid for $m\ge3$, $n\ge3$.
 \end{theorem}

In other words, for a generic Lie bracket on $\n$ with the bi-dimensions in the range $m,n\ge3$
the automorphism group of the Carnot structure $(\exp\n,\n_{-1})$ is a Lie group.
 % i.e. the symmetry algebra finite-dimensional
Before we give a proof of this, let us notice that several authors have studied automorphisms of generic
2-step Carnot structures. In particular, let us mention the result by P.\,Pansu~\cite[Proposition 13.1]{Pa},
asserting that in general, the automorphism group is generated by translations $\op{ad}_x$, $x\in\n_{-1}$,
and the standard homothety, namely the action by the grading element $e\in\n_0\subset\hat\n$,
provided that $n\in2\mathbb{Z}$, $n\ge10$, $3\le m<2n-3$.
The restrictions on bi-dimensions were relaxed by P.\,Eberlein~\cite[Proposition 3.4.3]{E2},
where his assumption is that $\bar{d}=\min\{m,\binom{n}2-m\}\ge3$ excluding the cases $n\le 6$
for $\bar{d}=3$. %(the conclusion was a bit weaker, but the same for the connected component of unity).
See also~\cite{KT,Riehm,Saal}.

When the stability subgroup of the automorphism group is only scaling due to the grading
element, the positive part of the Tanaka prolongation vanishes.
Indeed, if $\n_0=\langle e\rangle$, where $e$ is the grading element and
$v\in\n_1$ is a non-zero element, then we can choose a codimension 1
subspace $\Pi\subset\n_{-1}$ and a vector $x\in\n_{-1}\setminus\Pi$ such
that $[x,v]=e$, $[\Pi,v]=0$. Then $[[x,y],v]=-y$,
$[[[x,y_1],y_2],v]=-[y_1,y_2]=0$ and $[[[x,y],x],v]=-3[x,y]=0$ for all
$y,y_1,y_2\in\Pi$, so that $\n_{-2}=[\n_{-1},\n_{-1}]=0$. 
This contradiction yields the claim of our theorem. 

However, in order to overcome the restrictions assumed by Pansu and Eberlein, we will give the
proof valid for all pairs $(n,m)$ with the range specified. Outside this range all the algebras $\n\in N(m,n)$
are of infinite type.

Notice that absence of rank 1 elements in the family $\{J_z:0\neq z\in\n_{-2}\}\subset \op{End}(\n_{-1})$
is a Zariski open condition, so proving there exists one rigid Lie algebra structure on $\n$ in
bi-dimensions $(m,n)$ implies the same for a generic one.

 \begin{proof}
We will give two proofs of Theorem~\ref{th:3}. First, let us remind that we work in the complexification.
Using criterion~\eqref{Lcriterion}, we have to show that generically,
for a basis $z_1\,\dots,z_m$ of $\n_{-2}$, the vectors $J_{z_1}x,\dots,J_{z_m}x$
span the space of dimension greater than $1$ for all non-zero vectors $x\in\n_{-1}$.
%(by Remark \ref{RK1} the criterion does not depend on the choice of metric, which we can assume positive definite).
In other words, choosing a basis in $\n_{-1}$, the $n\times m$ matrix $M=[J_{z_1}x|\dots|J_{z_m}x]$ has rank less than or equal to $1$ only if $x=0$. In the case $x=0$ the rank is indeed 0.

The first proof is constructive. The condition $\rm{rank}\,(M)\le1$ means that
all $2\times 2$ minors vanish. Take independent minors $M_{1,i}^{1,j}$ for $1<i\le n$, $1<j\le m$.
The entries (linear in $x$) are at our disposal since we are free to perturb the operators $J_{z_j}$,
so we get $(n-1)(m-1)>n$ quadratic conditions on $x\in\n_{-1}$ whose common solution is generically
only zero, thus proving our claim.

For example, in the case $(m,n)=(3,5)$ we can have the following explicit matrices
giving the structure of $\n$:
 \begin{equation*}
\!\!J_{z_1}=\begin{bmatrix}
0 & 1 & 0 & 0 & 1\\
-1& 0 & 0 & 0 & 0\\
0 & 0 & 0 & 1 & 0\\
0 & 0 &-1 & 0 & 0\\
-1& 0 & 0 & 0 & 0\end{bmatrix},
J_{z_2}=\begin{bmatrix}
0 & 0 & 0 & 0 & 0\\
0 & 0 & 1 & 0 & 1\\
0 &-1 & 0 & 0 & 0\\
0 & 0 & 0 & 0 & 0\\
0 &-1 & 0 & 0 & 0\end{bmatrix},
J_{z_3}=\begin{bmatrix}
0 & 0 & 0 & 1 & 0\\
0 & 0 & 0 & 0 & 1\\
0 & 0 & 0 & 0 & 0\\
-1& 0 & 0 & 0 & 1\\
0 &-1 & 0 &-1 & 0\end{bmatrix}
 \end{equation*}
Then for $0\neq x=(x_1,x_2,x_3,x_4,x_5)^t\in\n_{-1}$ the $5\times3$ matrix
 $$
[J_{z_1}x|J_{z_2}x|J_{z_3}x]=
\begin{bmatrix}
x_2+x_5 & 0 & x_4\\
-x_1 & x_3+x_5 & x_5\\
x_4 & -x_2 & 0 \\
-x_3 & 0 & x_5-x_1\\
-x_1 & -x_2 & -x_2-x_4\end{bmatrix}
 $$
has rank greater or equal than $2$, whence this Lie algebra, as well as generic nilpotent algebras $\n$
with $(m,n)=(3,5)$, are of finite type.

The second proof uses the fact that there exist three linearly independent operators $J_1,J_2,J_3\in\mathfrak{so}(n)$
such that any 2-dimensional subspace of their span generates the whole Lie algebra $\mathfrak{so}(n)$. Note that any 3 generic skew-symmetric operators satisfy this property.
Statements of this kind can be found in~\cite{BG}. Assume that $J_{\n_{-2}}$ contains
three operators $J_{z_1},J_{z_2},J_{z_3}$ of the indicated type. Then the intersection $L^\dagger$
of the three-dimensional span$\{ z_1,z_2,z_3\}\subset\n_{-2}$ with the hyperplane $L$,
used in criterion~\eqref{Lcriterion}, has dimension 2 or 3.
Since $L^\dagger\subset L$, there exists a non-zero vector $x\in\n_{-1}$ satisfying $J_zx=0$ for
all $z\in L^\dagger$.
However $L^\dagger$ generates the Lie algebra $\mathfrak{so}(n)$ and thus we get $Ax=0$ for all
$A\in\mathfrak{so}(n)$, hence $x=0$. This contradiction proves the result.
 \end{proof}

It is natural to investigate the moduli space of nilpotent Lie structures on $\n=\n_{-2}\oplus\n_{-1}$,
i.e., the quotient of $N(m,n)$ by the natural action of $\op{GL}(n)$ on $\n$,
induced by the action on $\n_{-1}$.
This is no longer a manifold due to existence of singular orbits.
However since the action is algebraic, it allows a rational quotient on a Zariski open stratum by Rosenlicht's theorem \cite{Ros}. Thus the quotient is a rational space and it has positive dimension in the following cases:
 \begin{itemize}
\item $\bar{d}=\min\{m,\binom{n}2-m\}\ge3$, $n\ge6$, because
    $\dim\mathfrak{sl}(n)<\dim\op{Gr}_m(\mathfrak{so}(n))$. The standard homothety acts trivially on the Grassmanian, so we consider its quotient by the projective group $\op{PSL}(n)$.
\item $\bar{d}=2$, $n\in2\mathbb{Z}$, $n>6$, because of the following. The algebra structure on $\n$ is
    given by two skew-symmetric operators $J_{z_1},J_{z_2}$ on $\n_{-1}$, which are generically non-degenerate. They are however not
    invariants of the Lie algebra structure, but only of the chosen $M$-algebra: under the change
    of metrics on $\n_{-2}$ and $\n_{-1}$ given by symmetric matrices $B$ and $A$ of sizes
    $2\times2$ and $n\times n$ respectively, the operator $J_z$ changes to $AJ_{Bz}$. Passage to
    $I=J_{z_1}^{-1}J_{z_2}$ eliminates dependence on $A$, and the action of the M\"obius group
    $I\mapsto \frac{a+bI}{c+dI}$ eliminates dependence on $B$. Since the spectrum of $I$ generically
    consists only of double eigenvalues with totality $|\op{Sp}(I)|=n/2$, we obtain a continuous invariant
    for $n>6$.
 \end{itemize}
In all other cases, there is an open orbit, and thus, no moduli for 2-step structures $\n\in N(m,n)$.
This is obvious if $\bar{d}\le1$. In the remaining cases,
the codimension of the orbit of $\op{PSL}(n)$ on $\op{Gr}_m(\mathfrak{so}(n))$ is
 \begin{equation}\label{codim}
\dim\op{Gr}_m(\mathfrak{so}(n))-\dim\mathfrak{sl}(n)+d(m,n)=
m\cdot\tbinom{n}2-m^2-n^2+1+d(m,n),
 \end{equation}
where $d(m,n)$ is the dimension of the stabilizer of a generic point $Z$ in the Grassmannian, or  equivalently, the minimum of dimensions of stabilizers of all points, which was computed in~
\cite[Section 4.3]{E2}:
 $$
d(2,2k+1)=2k+4,\ d(2,4)=7,\ d(2,6)=9,\ d(3,4)=6,\ d(3,5)=3.
 $$
It is straightforward to check that the value in~\eqref{codim} is zero for these bi-dimensions, and that it is positive for $(2,2k)$ since $d(2,2k)=3k$ for $k>3$.

We conclude generic rigidity for the moduli on the strata of highest dimension in
the quotient of $N(m,n)$ by $\op{PSL}(n)$ in the case $m>2$.

%%%%%%%%%%%%%%%%%%%%%%%%%

\section{Digression: rigidity vs.\ pseudo $H$-type}\label{sec:add}

Let us at first discuss the rigidity problem in lowest dimensions.
The first non-trivial case is thus $\dim\n_{-2}=3$.
The fundamental property implies that $\dim\n_{-1}\ge3$ and in the case of equality
$\dim\n_{-2}=\dim\n_{-1}=3$ the bracket identifies $\n_{-2}=\Lambda^2\n_{-1}$.
It is well-known, see~\cite{CS,Y}, that the Tanaka prolongation of $\n$ in this case is the simple Lie
algebra of type $B_3$ with the 2-grading induced by the parabolic subalgebra $\mathfrak{p}_3$. Namely, it is $\mathfrak{so}(3,4)$ in the real case or $\mathfrak{so}(7,\mathbb{C})$ in
the complex case. Thus the algebra $\n$ with $\dim\n_{-2}=\dim\n_{-1}=3$ is rigid.

The situation changes when $\dim\n_{-2}=3$ and $\dim\n_{-1}=4$. The classification of such graded 2-step Lie algebras,
which can be extracted from~\cite{Ku}, is as follows. Let $e_1,e_2,e_3,e_4$ be a basis of $\n_{-1}$,
and let $f_1,f_2,f_3$ be that of $\n_{-2}$. The non-trivial brackets of $\n$ in every of six non-isomorphic
cases are given below. These cases are numerated in loc.cit.\ as $m7\_2\_2$, $m7\_2\_3$, $m7\_2\_4$, $m7\_2\_2r$,
$m7\_2\_5$ and $m7\_2\_5r$ respectively:
 \begin{gather}
[e_1,e_4]=f_1,\ [e_2,e_4]=f_2,\ [e_3,e_4]=f_3; \label{GNLA1}  \\
[e_1,e_4]=f_1,\ [e_2,e_4]=f_2,\ [e_3,e_4]=f_3, [e_2,e_3]=f_1; \label{GNLA2} \\
[e_1,e_4]=f_1,\ [e_2,e_3]=f_2,\ [e_3,e_4]=f_3; \label{GNLA3} \\
[e_1,e_3]=[e_4,e_2]=f_1,\ [e_1,e_4]=[e_2,e_3]=f_2,\ [e_3,e_4]=f_3; \label{GNLA4} \\
[e_1,e_2]=[e_3,e_4]=f_1,\ [e_1,e_4]=f_2,\ [e_2,e_3]=f_3; \label{GNLA5} \\
[e_1,e_2]=[e_3,e_4]=f_1,\ [e_1,e_3]=[e_4,e_2]=f_2,\ [e_1,e_4]=[e_2,e_3]=f_3. \label{GNLA6}
 \end{gather}
Clearly in cases~(\ref{GNLA1}--\ref{GNLA3}) $\op{rank}(\op{ad}_{e_1})=1$, and
for \eqref{GNLA4} we have to use the complexification $\op{rank}(\op{ad}_{e_1+ie_2})=1$,
so the corresponding algebras $\n$ are of infinite type by the rank 1 criterion \cite{DR,O}.
On the contrary, in cases \eqref{GNLA5} and \eqref{GNLA6} the algebras $\n$ are rigid.
Indeed, in these cases they are isomorphic to the pseudo $H$-type algebras $\n^{1,2}$ and $\n^{3,0}$ respectively. Thus, we conclude the following.
 \begin{proposition}
The algebra $\n$ with $(\dim\n_{-2},\dim\n_{-1})=(3,4)$ is of finite type if and only if it is of
pseudo $H$-type.
 \end{proposition}
The corresponding statement does not hold for $n=\dim\n_{-1}>4$. Already for
$(m,n)=(3,5)$ we have rigid algebras $\n\in N(m,n)$ that are not of pseudo $H$-type.
A criterion for 2-step nilpotent Lie algebras to be of $H$-type was obtained in~\cite{LT}.
Indeed,   there are no 5-dimensional representations without trivial modules for the Clifford algebra generated by a 3-dimensional scalar product space, see Table~\ref{t:dim} in Section~\ref{sec:J2Htype}. More generally, the following holds
 \begin{theorem}
Generic (resp.\ generic rigid) algebras $\n\in N(m,n)$ in the range $m>1$ (resp.\ $m>2$)
except for $(m,n)\in\{(2,4),(3,4)\}$ are not of pseudo $H$-type.
 \end{theorem}

 \begin{proof}
Any representation of the Clifford algebra $\op{Cl}(\n_{-2},\langle\cdot,\cdot\rangle_{\n_{-2}})$ is decomposed
into a direct sum of irreducible modules. We have to exclude trivial submodules which lead to infinite type.
Thus, first of all, not every pair  of bi-dimensions $(m,n)$ can be realized for a pseudo $H$-type algebra.

Second, if $\n_{-1}=\oplus_{i=1}^k U_i$ is the sum of $k$ irreducible modules, then $k$ scalings
contribute to the choice of module, and hence, to fixing  the scalar product on $\n$:
$\langle\cdot,\cdot\rangle=\langle\cdot,\cdot\rangle_{\n_{-2}}+\langle\cdot,\cdot\rangle_{\n_{-1}}$.

Next, changing the scalar product we keep the same Lie structure of $\n$ but a different $J$-representation
as an $M$-type algebra. Namely, any other scalar product, having decomposition
$\n=\n_{-2}\oplus\n_{-1}$ orthogonal, can be obtained via two symmetric, not necessarily positive definite,
operators $A\in\op{End}(\n_{-1})$ and $B\in\op{End}(\n_{-2})$:
$\langle\cdot,\cdot\rangle_\text{new}=\langle A.\,,.\rangle_{\n_{-1}}+\langle B.\,,.\rangle_{\n_{-2}}$.
The $J$-representation of the Lie algebra
structure changes so $J_z\rightsquigarrow A\circ J_{Bz}$. The change of the basis $z\mapsto Bz$ in $\n_{-2}$
does not influence the dimension of the space of pseudo $H$-type algebras in $N(m,n)$, while the symmetric operator $A\in \op{End}(\n_{-1})$ does contribute to it.

Alternatively, we can think of $\op{GL}(n)$ changing the basis in $\n_{-1}$, inducing the change in $\n_{-2}$.
But since the orthogonal group preserves $\langle\cdot,\cdot\rangle_{\n_{-1}}$, we obtain only $\binom{n+1}2$
transformations, including the above $k$ scalings, and this number  bounds from above the
dimension of the space of $H$-type algebras.

For $n>4$, $1<m<\binom{n}2-1$, we have $\binom{n+1}2<md=\dim\op{Gr}_d(\mathfrak{so}(n))$,
whence a generic $\n\in N(m,n)$ in this range of bi-dimensions is not of pseudo $H$-type.

For $n>2$ it can be checked that in the cases $m=\binom{n}2$ the only Lie structure
(that is the modelled on the simple algebra $B_n$ with the grading corresponding to $\mathfrak{p}_n$)
is not of pseudo $H$-type, as well as all finite number of structures $\n$ for $m=\binom{n}2-1$
are not such. This follows by dimensional reasons similarly to the bi-dimension $(3,5)$ considered before the theorem. Thus, we only have to study the cases $(m,n)$ with $3\le n\le 4$, which do not satisfy the above inequalities.

They are precisely the cases $(2,4)$, $(3,4)$ and $(4,4)$. The last one may not have any pseudo
$H$-type as it follows from Table \ref{t:dim} in the next section. But the first two both admit a pseudo
$H$-type structure and they are two exceptions:  in the first case $\n$ is always of infinite type and in the second case it is generically of finite type. This finishes the proof.
 \end{proof}

An alternative approach to the proof of the above theorem for $\dim\n_{-2}>2$ is as follows.
For  pseudo $H$-type algebras, almost all operators $J_{z_i}$ are invertible and one can consider
the operators $T_{ij}=J_{z_i}^{-1}J_{z_j}$ for some generic choice of the basis $z_i\in\n_{-2}$.
Similarly to how it was done in \cite{LT} for $H$-type algebras, one can show that $T_{ij}$ generate
a subalgebra of the even part $\op{Cl}_0(\n_{-2},\langle .\,,.\rangle_{\n_{-2}})$ of the Clifford algebra.
In particular, this subalgebra has dimension at most $p=2^{m-1}$. Since the minimal dimension of
the Clifford module is $2^{\frac{m-1}2}\le n$, we conclude that $p\le n^2=\dim\op{End}(\n_{-1})$, where
the inequality is strict unless $n\equiv7\!\!\mod8$. Due to the results of \cite{BG}
the operators $T_{ij}$, obtained from generic operators $J_k$ as above, % associatively
generate the whole endomorphism algebra $\op{End}(\n_{-1})$.
Therefore, the claim follows from the inequality $p<n^2$ for $n\not\equiv7\!\!\mod8$,
and by comparison of the structures of $\op{Cl}_0(\n_{-2}, \langle .\,,.\rangle_{\n_{-2}})$ with $\op{End}(\n_{-1})$ otherwise.

%%%%%%%%%%%%%%%%%%%%%%%%%

\section{Pseudo $H$-type algebras with $J^2$-condition}\label{sec:J2Htype}

In this section we clasify pseudo $H$-type algebras satisfying the $J^2$-condition.

 \begin{definition}\label{def:admis_mod}
Let $J\colon\Cl(U,\langle\cdot,\cdot\rangle_{U})\to\End(V)$ be a Clifford algebra representation. The module $V$ is called admissible if there is a bilinear form $\langle\cdot,\cdot\rangle_{V}$ such that  the endomorphisms $J_z$ are skew-symmetric for any $z\in U$, that is
\begin{equation}\label{eq:skew}
\langle J_zx,y\rangle_{V}=-\langle x,J_zy\rangle_{V}.
\end{equation}
This scalar product $\langle\cdot,\cdot\rangle_{V}$ is called admissible for the module~$V$.
 \end{definition}

It is well-known~\cite{H,Lam} that if $U$ is endowed with a positive definite bilinear form $\langle\cdot,\cdot\rangle_{U}$, then the module $V$ is admissible with respect to some positive definite bilinear form
$\langle\cdot,\cdot\rangle_{V}$. In particular, any irreducible module is admissible. In the case when the bilinear form $\langle\cdot,\cdot\rangle_{U}$ is indefinite, the module $V$ is not always admissible and sometimes only the direct sum $V\oplus V$ is admissible. As a consequence, in this case not all irreducible modules are admissible. Notice also that  if $(U,\langle\cdot,\cdot\rangle_{U})$ is an indefinite scalar product space, then any admissible module will necessarily be a neutral space~\cite{Ci}. We call an admissible module of minimal possible dimension the {\it minimal admissible module}.

In Table~\ref{t:dim} we give the dimensions of the minimal admissible modules $V^{r,s}$, $r,s\leq 8$. Dimensions of other minimal admissible modules can be obtained by Bott's periodicity, see \cite{LM}. The bold integers are used for the minimal admissible modules which are direct sums of two irreducible Clifford modules. Others denote the dimensions of minimal admissible modules, that are also irreducible Clifford modules. The notation ${}_{\times 2}$ means that there are two minimal admissible modules, related to non-isomorphic irreducible modules.

\begin{table}[h]
\center\caption{Dimensions of minimal admissible modules}
\begin{tabular}{|c||c|c|c|c|c|c|c|c|c|}
\hline
${\text{\small 8}} $&$ {\text{\small{16}}}$&${\text{\small 32}}$&${\text{\small{64}}}
$&${\text{\small{64}$_{\times 2}$}}$&${\text{\small{128}}}$&${\text{\small{128}}}$&${\text{{\small{128}}}}$&$
{\text{{\small{128}$_{\times 2}$}}} $&${\text{\small{256}}}$
\\
\hline
${\text{\small 7}}$ &$ {\text{\small{16}}}$&${\text{\small{32}}}$&$
{\text{\small{\bf{64}}}} $&${\text{\small{64}}}
$&${\text{{\bf\small{128}}}}$&${\text{\small{{\bf{128}}}}}  $&${\text{\bf{\small{128}}}}$&$ {\text{\small{128}}} $&${\text{\small{256}}}$
\\
\hline
${\text{\small 6}}$ &${\text{\small{16}}}$&${\text{\small{16}$_{\times 2}$}}$&${\text{\small{32}}}$&${\text{\small{32}}}$&${\text{\small{\bf{64}}}}
$&${\text{\bf{\small{64}$_{\times 2}$}}} $&${\text{\bf\small{128}}} $&${\text{\small{128}}}$&$ {\text{\small{256}}} $
\\
\hline
${\text{\small 5}} $&${\text{\small\bf {16}}}$&${\text{\small 16}}$&${\text{\small 16}}$&${\text{\small 16}}$&${{\text{\small\bf 32}}}$&${{\text{\small \bf{64}}}} $&${\text{\small{\bf{128}}}}$&${\text{\small{128}}} $&$\text{\small{\bf 256}}$
\\
\hline
${\text{\small 4}} $&$  {\text{\small 8}}$&$ {\text{\small 8}}$&$
{\text{\small 8}}$&$ 8_{\times 2}$&$16$&${\text{\small 32}}$&${\text{\small 64}}
$&${\text{\small 64}_{\times 2}} $&${\text{\small{128}}}$
\\
\hline
${\text{\small 3}}$&${{\text{\small\bf 8}}}$&${{\text{\small\bf 8}}}$&${\text{\small\bf 8}}$&$8$&$16$&$32$
&${\text{\small\bf 64}}$&$64$&${{\text{\small\bf 128}}}$
\\
\hline
${\text{\small 2}}$&${{\text{\small\bf 4}}}$&$
{\bf 4_{\times 2}}$&${\bf 8}$&$ 8$&$16$&$16_{\times 2}$&$32$&$32 $&${{\text{\small\bf 64}}}$
\\
\hline
${\text{\small 1}}$ &${\bf 2}$&${\bf 4}$&${\bf 8}$& $8$&${\bf 16}$&$16$&$16$&$16$&${{\text{\small\bf 32}}}$
\\
\hline
${\text{\small 0}} $&$  1$&$ 2$&$ 4$&$ 4_{\times 2}$&$ 8$&$ 8$&$ 8$&$ 8_{\times 2}$&$16$
\\
\hline\hline
{$s$/$r$}&  {\text{\small 0}}& {\text{\small 1}}&
{\text{\small 2}}&{\text{\small 3}} & {\text{\small 4}}& {\text{\small 5}}& {\text{\small 6}}& {\text{\small 7}}& {\text{\small 8}}
\\
\hline
\end{tabular}\label{t:dim}
\end{table}

As it was mentioned before, pseudo $H$-type Lie algebras are closely related to Clifford algebras. Namely, for a pseudo $H$-type Lie algebra $\n=(\n_{-2}\oplus\n_{-1},\langle\cdot,\cdot\rangle_{\n_{-1}}+\langle\cdot,\cdot\rangle_{\n_{-2}})$ one has the representation
$ J\colon\Cl(\n_{-2}, \langle\cdot,\cdot\rangle_{\n_{-2}})\to\End(\n_{-1})$.
Conversely, for an admissible $\Cl(U,\langle\cdot,\cdot\rangle_{U})$-module $(V,\langle\cdot,\cdot\rangle_V)$, the representation induces a graded 2-step nilpotent Lie algebra structure on $\n_{-2}\oplus\n_{-1}=U\oplus V$ defining the Lie bracket by equation~\eqref{def:J}.

Let us assume that $\n$ satisfies the $J^2$-condition. If $x\in\n_{-1}$ and $z,z'\in\n_{-2}$ is any orthogonal pair satisfying
\begin{equation}\label{eq:forall}
\la J_{\tilde z}x,J_zJ_{z'}x\ra_{\n_{-1}}=0,\qquad \ \forall\ \ \tilde z\in\n_{-2},
\end{equation}
then $\la x,x\ra_{\n_{-1}}=0$. Indeed using \eqref{JJxx} we get
$$
0=\la J_{\tilde z}x,J_zJ_{z'}x\ra_{\n_{-1}}=\la J_{\tilde z}x,J_{z''}x\ra_{\n_{-1}}=\la \tilde z,z''\ra_{\n_{-2}}\la x,x\ra_{\n_{-1}}.
$$
Thus, to show that a pseudo $H$-type Lie algebra $\n$ does not satisfy the $J^2$-condition it is enough to find a vector $x\in\n_{-1}$, $\langle x,x\rangle_{\n_{-1}}\neq 0$ and an orthogonal pair $z,z'\in\n_{-2}$ such that~\eqref{eq:forall} holds.

For a minimal admissible module $V$ let us call the modules $V^{\oplus k}$ isotypic.

\begin{theorem}\label{th:J2}
Only the following pseudo $H$-type algebras satisfy the $J^2$-condition:
\begin{itemize}
\item[{(1)}] {$\dim\n_{-2}=0$: ${\mathbb R}^n$ -- any module (vector space) over $\mathbb{R}$.}
\item[{(2)}] { $\dim\n_{-2}=1$: $\n^{1,0}$ and $\n^{0,1}$ for any admissible module.}
\item[{(3)}] { $\dim\n_{-2}=3$: $\n^{3,0}$ and $\n^{1,2}$ for any isotypic module.}
\item[{(4)}] { $\dim\n_{-2}=7$: $\n^{7,0}$ and $\n^{3,4}$ for the minimal admissible modules.}
\end{itemize}
\end{theorem}

In Appendix \ref{Ap.A} we will explain these dimensions via a relation to the division algebras.
Explicit descriptions of the admissible modules are given in the proof.
They are also realized via simple Lie algebras as described in Appendix \ref{Ap.B}.

\begin{proof}
We start by presenting a dimensional argument related to Table~\ref{t:dim}, which shows that all pseudo $H$-type algebras $\n^{r,s}$ with $r+s\neq0,1,3,7$, do not satisfy the $J^2$-condition. Let us start by pointing out that if $\n^{r,s}(V)$, with a minimal admissible module  $V$, does not satisfy the $J^2$ condition, then no $\n^{r,s}$ satisfies it. Indeed, for any pseudo $H$-type algebra $\n^{r,s}(V)$, with a minimal admissible module $V$, there exists an element $x\in\n_{-1}=V$, $\langle x,x\rangle_{\n_{-1}}=1$, such that the set
\[
\{x,J_{z_i}x,J_{z_i}J_{z_j}x,J_{z_i}J_{z_j}J_{z_k}x,\cdots\}
\]
contains an orthonormal basis of $\n_{-1}$, see~\cite{FM}.

Since ${\mathbb R}x\oplus J_{\n_{-2}}x$ is an admissible module, it cannot have dimension less than the dimension of the minimal admissible module listed in Table~\ref{t:dim}. From this we see that the Clifford algebra ${\rm Cl}(\n_{-2},\langle\cdot,\cdot\rangle_{\n_{-2}})$ can possess an admissible module of the form ${\mathbb R}x\oplus J_{\n_{-2}}x$ of dimension $r+s+1$ only when $r+s=0,1,3,7$. In other words, if $r+s\neq0,1,3,7$, no pseudo $H$-type algebra $\n^{r,s}$ admits the $J^2$-condition.

Henceforth we focus on pseudo $H$-type algebras $\n^{r,s}$ with $m=r+s=0,1,3,7$.
Fix an orthonormal basis $\{z_1,\dotsc,z_{r+s}\}$ of the center $\n_{-2}$ of Lie algebra $\n^{r,s}$ with
$\la z_i,z_i\ra_{\n_{-2}}=1$, $1\leq i\leq r$ and $\la z_i,z_i\ra_{\n_{-2}}=-1$, $r<i\leq r+s$.

The $J^2$-condition is trivially satisfied when $m=\dim\n_{-2}=0,1$. Notice that for $\n^{1,0}$ or
$\n^{0,1}$ the admissible modules are isotypic $V^{\oplus k}$, where $V$ is the minimal admissible
module $\mathbb{C}$ or $\mathbb{R}\oplus\mathbb{R}$ respectively.

Let $\dim\n_{-2}=3$. A similar dimensional argument as before, using Table~\ref{t:dim}, shows that the cases $(r,s)\in\{(2,1),(0,3)\}$ do not satisfy the $J^2$-condition. For $(r,s)\in\{(3,0),(1,2)\}$, we obtain
$
(J_{z_1}J_{z_2}J_{z_3})^2=
\Id_{\n_{-1}}$.
In these cases there are two non-isomorphic irreducible 2-dimensional Clifford modules $V_+$ and $V_-$ of $\n_{-1}$, where the endomorphism $\Omega^{r,s}=J_{z_1}J_{z_2}J_{z_3}$ acts as the identity or minus the identity, respectively. In other words, the spaces $V_{\pm}$ are the eigenspaces of $\Omega^{r,s}$ with the eigenvalues $\pm 1$. In both cases  $(r,s)\in\{(3,0),(1,2)\}$ the dimension of the minimal  admissible module is $4$.

Case $(r,s)=(1,2)$. The minimal admissible module is either $V_+\oplus V_{+}$ or $V_-\oplus V_-$, see~\cite{FM,FM2}. It is necessary to point out that each module $V_+$ and $V_-$ is a null space and by this reason we need to double them to guarantee admissibility. It follows that neither of these direct sums is orthogonal.
For an orthonormal basis $\{z_1,z_2,z_3\}$ of $\n_{-2}$, it holds
\begin{equation}\label{eq:12}
\left\{\begin{array}{l}
 J_{z_1}x=-J_{z_2}J_{z_3}x,
\\
 J_{z_2}x=-J_{z_1}J_{z_3}x ,
\\
J_{z_3}x=J_{z_1}J_{z_2}x,
\end{array}\right.
\end{equation}
for all $x\in V_{+}$. The signs in \eqref{eq:12} change to the opposite ones for $x\in V_{-}$. Consider $\n^{1,2}(V_+\oplus V_{+})=\n_{-2}\oplus\n_{-1}$. Note that we could have chosen $\n^{1,2}(V_-\oplus V_{-})$ instead, but we would have obtained an isomorphic algebra.
%Thus we are working for the moment in minimal admissible modules.
Choose any $z=\sum_{i=1}^{3}a_iz_i$ and $z'=\sum_{i=1}^{3}b_iz_i$ such that $\langle z,z' \rangle_{\n_{-2}}=0$. Then
\begin{equation}\label{eq:exist12}
J_zJ_{z'}x
=\sum_{i,j}a_ib_jJ_{z_i}J_{z_j}x=\langle z,z' \rangle_{\n_{-2}}\Id_{\n_{-1}}x+\sum_{i\neq j}a_ib_jJ_{z_i}J_{z_j}x
=J_{z''}x,
\end{equation}
where $z''$ is a linear combination of the orthonormal basis with coefficients $\pm a_ib_j$, $i\neq j$ with the signs obtained by using~\eqref{eq:12}. We conclude that the pseudo $H$-type algebra $\n^{1,2}(V_+\oplus V_{+})$ satisfies the $J^2$-condition.

In the next step we allow the admissible module $\n_{-1}$ to be non-minimal.
There are (potentially several) non-isomorphic pseudo $H$-type algebras $\n^{1,2}(V^{p,q})$ where
$
\n_{-1}=V^{p,q}=(V_+\oplus V_+)^{\oplus p}\oplus(V_-\oplus V_-)^{\oplus q}$
is an orthogonal sum of minimal admissible modules. It is not hard to see that the Lie algebras $\n^{1,2}(V^{p,0})$ and $\n^{1,2}(V^{0,q})$ satisfy the $J^2$-condition. We refer to these cases as {\it isotypic}.

At last we claim that for $p,q>0$ the pseudo $H$-type Lie algebra $\n^{1,2}(V^{p,q})$ does not satisfy the $J^2$-condition. Since this property for $(p,q)$ implies the one for $(p',q')$ with $p'\ge p,q'\geq q$, 
it is enough to check this for $p=q=1$.

Choose $x_\pm\in V_\pm\oplus V_\pm$ such that 
$\langle x_+,x_+\rangle_{\n_{-1}}=\langle x_-,x_-\rangle_{\n_{-1}}\neq 0$.
Then $x=x_++x_-\in\n_{-1}$ is such that 
$\langle x,x\rangle_{\n_{-1}}=2\langle x_+,x_+\rangle_{\n_{-1}}\neq 0$ and we get
  \begin{eqnarray*}
\langle J_{z_1}J_{z_2}x,J_zx\rangle_{\n_{-1}}
&=&
a_3\langle J_{z_1}J_{z_2}x,J_{z_3}x\rangle_{\n_{-1}}
\\
&=&
a_3\langle J_{z_3}x_+,J_{z_3}x_+\rangle_{\n_{-1}}-a_3\langle J_{z_3}x_-,J_{z_3}x_-\rangle_{\n_{-1}}=0.
 \end{eqnarray*}

The case $(r,s)=(3,0)$ is well known, see~\cite{Ci,CDKR}, but still can be treated in a similar manner as before, changing~\eqref{eq:12} to the following ordered product
$$
\pm J_{z_i}x=(-1)^{i}\prod_{j\neq i}J_{z_j}x\quad\text{for all} \quad x\in V_{\pm}.
$$

Let $\dim\n_{-2}=7$. A similar dimensionality argument as before, using Table~\ref{t:dim}, shows that the cases $(r,s)\notin\{(7,0),(3,4)\}$ do not satisfy the $J^2$-condition.

Case $(r,s)=(3,4)$. There are two irreducible Clifford modules $V_+$ and $V_-$ of dimension 8 that, in this case, are also admissible. They are isometric to $\mathbb R^{4,4}$ and generate isomorphic pseudo $H$-type algebras, see~\cite{FM,FM2}.
The following operators are mutually commuting symmetric involutions:
$$
P_1=J_{z_1}J_{z_2}J_{z_4}J_{z_5},\quad P_2=J_{z_1}J_{z_2}J_{z_6}J_{z_7},\quad P_3=J_{z_1}J_{z_3}J_{z_5}J_{z_7}.
$$
Table~\ref{t:34} shows the commutation relations of the operators $P_j$ and $J_{z_k}$,
where $1$ appears if they commute and $-1$ if they anti-commute.
\begin{table}[h]
\center\caption{Commutation relations of operators}
\begin{tabular}{|c||c|c|c|c|c|c|c|c|c|c|c|c|c|c|}
\hline
\ &$J_{z_1}$&$J_{z_2}$&$J_{z_3}$&$J_{z_4}$&$J_{z_5}$&$J_{z_6}$&$J_{z_7}$
\\
\hline\hline
$P_1$&$-1$&$-1$&1&$-1$&$-1$&1&1
\\
\hline
$P_2$&$-1$&$-1$&1&1&1&$-1$&$-1$
\\
\hline
$P_3$&$-1$&1&$-1$&1&$-1$&1&$-1$
\\
\hline\hline
\end{tabular}\label{t:34}
\end{table}

We observe that $V_+$ decomposes into the orthogonal sum of the common eigenspaces for
the involutions $P_1,P_2,P_3$. Denote $E^{1,1,1}=\{w\in\n_{-1}\,|\,P_1w=P_2w=P_3w=w\}$.
Then all $J_{z_i}E^{1,1,1}$ are different mutually orthogonal eigenspaces for $P_1$, $P_2$, $P_3$.
Therefore, for any pair $J_{z_i}J_{z_j}$, $i\neq j$, there is $k\neq i,j$ such that $J_{z_i}J_{z_j}(w)=\pm J_{z_k}(w)$. Table~\ref{t:vect} shows these eight different eigenspaces.
{\tiny
\begin{table}[h]
\begin{center}
\caption{Eigenspace decomposition: $\Cl_{3,4}$ case}
\begin{tabular}{|l|c|c|c|c|c|c|c|c|} \hline
Involutions&\multicolumn{8}{|c|}{Eigenvalues}\\ \hline
$P_1$&\multicolumn{4}{|c|}{$+1$} &\multicolumn{4}{|c|}{$-1$}\\\hline
$P_2$&\multicolumn{2}{|c|}{$+1$}
&\multicolumn{2}{|c|}{$-1$}&\multicolumn{2}{|c|}{$+1$}
&\multicolumn{2}{|c|}{$-1$}\\\hline
$P_3$&\multicolumn{1}{|c|}{$+1$}
&\multicolumn{1}{|c|}{$-1$}&\multicolumn{1}{|c|}{$+1$}
&\multicolumn{1}{|c|}{$-1$}
&\multicolumn{1}{|c|}{$+1$} &\multicolumn{1}{|c|}{$-1$}
&\multicolumn{1}{|c|}{$+1$} &\multicolumn{1}{|c|}{$-1$}\\
\hline\hline
{Eigenvectors}
&$w$&$J_{z_3}w$&$J_{z_6}w$& $J_{z_7}w$&$J_{z_4}w$&$J_{z_5}w$&$J_{z_2}w$&$J_{z_1}w$
\\
&$\ $&$J_{z_1}J_{z_2}w$&$J_{z_1}J_{z_5}w$& $J_{z_1}J_{z_4}w$&$J_{z_1}J_{z_7}w$&$J_{z_1}J_{z_6}w$&$J_{z_1}J_{z_3}w$&$J_{z_2}J_{z_3}w$
\\
&$\ $&$J_{z_4}J_{z_5}w$&$J_{z_2}J_{z_4}w$& $J_{z_2}J_{z_5}w$&$J_{z_2}J_{z_6}w$&$J_{z_2}J_{z_7}w$&$J_{z_4}J_{z_6}w$&$J_{z_4}J_{z_7}w$
\\
&$\ $&$J_{z_6}J_{z_7}w$&$J_{z_3}J_{z_7}w$& $J_{z_3}J_{z_6}w$&$J_{z_3}J_{z_5}w$&$J_{z_3}J_{z_4}w$&$J_{z_5}J_{z_7}w$&$J_{z_5}J_{z_6}w$
\\\hline
\end{tabular}\label{t:vect}
\end{center}
\end{table}
}

Choosing the basis
$$
\begin{array}{lllllll}
& x_1=w,\quad & x_2=J_{z_3}w,\quad & x_3=J_{z_6}w,\quad & x_4=J_{z_7}w,
\\
&x_5=J_{z_4}w,\quad & x_6=J_{z_5}w,\quad & x_7=J_{z_2}w,\quad & x_8=J_{z_1}w,
\end{array}
$$
for any $x_p$ and any indices $i\neq j$, there is an index $k$ such that
$J_{z_i}J_{z_j}(x_p)=\pm J_{z_k}(x_p)$.
Therefore, for any $x=\sum_{p=1}^{8}\lambda_px_p$, we again obtain
$J_{z_i}J_{z_j}(x)=J_{\tilde z}(x)$ for some $\tilde z\in\n_{-2}$.
Now for any $x\in\n_{-1}$ and for any pair of orthogonal vectors $z,z'\in\n_{-2}$, we argue as in~\eqref{eq:exist12} and conclude that $\n^{3,4}$ with the minimal admissible module satisfies the $J^2$-condition.

Let us show that $\n^{3,4}(W)$ with a non-minimal admissible module $W$ does not satisfy the $J^2$-condition. It is clear that if the $J^2$-condition holds for $\n^{3,4}(W)$ it also holds for $\n^{3,4}(W')$,
where $W'\subset W$ is a submodule. Thus it suffices to verify the case of two summands
$W=V_1\oplus V_2$, where $V_i=V_+,V_-$. All three cases are similar and we consider only the first of them: $V_1=V_2=V_+$.

Let $x\in\n_{-1}$ be of the following form $x=w_{1}+J_{z_k}w_2$, where $w_1\in E^{1,1,1}\cap V_1$ and $w_2\in E^{1,1,1}\cap V_2$ and $k=1$ or $2$. We assume that the admissible metric is such that $\langle w_{\alpha},w_{\alpha}\rangle_{V_+}>0$, $\alpha=1,2$, and normalizing we assume that $w_1,w_2$ are unit vectors. This implies that $\langle x,x\rangle_{\n_{-1}}\neq 0$.  As a pair of orthogonal vectors from $\n_{-2}$ we take the basis vectors $z_3$ and $z_4$.
For an arbitrary vector $z=\sum_{i=1}^{7}a_iz_i$ the following holds
 \begin{eqnarray*}
&\ &\langle J_{z_3}J_{z_4}x,J_zx\rangle_{\n_{-1}}
=
\Big\langle J_{z_3}J_{z_4}(w_1+J_{z_k}w_2),\sum_{i=1}^{7}a_iJ_{z_i}(w_1+J_{z_k}w_2)\Big\rangle_{\n_{-1}}
\\
&=&
a_5\langle J_{z_3}J_{z_4}w_1,J_{z_5}w_1\rangle_{\n_{-1}}
+
\Big\langle J_{z_k}J_{z_3}J_{z_4}w_2,-\sum_{i\neq k}a_iJ_{z_k}J_{z_i}w_2\Big\rangle_{\n_{-1}}
\\
&=&
a_5\langle J_{z_3}J_{z_4}w_1,J_{z_5}w_1\rangle_{\n_{-1}}
-
a_5\langle z_k,z_k\rangle_{\n_{-2}}\langle J_{z_3}J_{z_4}w_2,J_{z_5}w_2\rangle_{\n_{-1}}=0
 \end{eqnarray*}
If the admissible metric $\langle .\,,.\rangle_{\n_{-1}}$ is such that $\langle w_{1},w_{1}\rangle_{\n_{-1}}=-\langle w_{2},w_{2}\rangle_{\n_{-1}}$, then we modify the previous argument by changing the vector $x\in\n_{-1}$ to $x=w_{1}+J_{z_k}w_2$, $k=5,6$ or $7$.
Thus only the algebra $\n^{3,4}(V_+)\simeq\n^{3,4}(V_-)$ satisfies the $J^2$-condition.

Case $(r,s)=(7,0)$. Though this case is known in the literature \cite{CDKR}, it can be treated similarly to the case $(3,4)$ starting from the involutions
$$
P_1=J_{z_1}J_{z_2}J_{z_3}J_{z_4},\quad P_2=J_{z_1}J_{z_2}J_{z_5}J_{z_6},\quad P_3=J_{z_1}J_{z_3}J_{z_5}J_{z_7}
$$
acting on the minimal admissible module $\n_{-1}=V_\pm$ isometric to $\mathbb R^8$.
The conclusion is very similar: only $\n^{7,0}(V_+)\simeq\n^{7,0}(V_-)$ satisfies the $J^2$-condition.
\qedhere
\end{proof}

%%%%%%%%%%%%%%%%%%%%%%%%%

\appendix

%%%%%%%%%%%%%%%%%%%%%%%%%

\section{Split versions of the division algebras}\label{Ap.A}

Recall that real rank of a Lie group $G$ is the dimension of the Abelian factor $A$ in the 
Iwasawa decomposition $G=KAN$. 
Real rank 1 simple Lie algebras are $|1|$- or $|2|$-graded
(i.e.\ $\g=\g_{-\nu}\oplus\dots\oplus\g_0\oplus\dots\oplus\g_\nu$, $\nu=1\vee2$), 
and the Lie algebra 
$\n=\op{Lie}(N)=\g_{-2}\oplus\g_{-1}=\z\oplus V$
is of $H$-type satisfying the $J^2$ condition \cite{Ko2}. 

It was noticed by B.\,Kostant \cite{Ko1} that in this case $\R\oplus\z$ is necessarily a division algebra
(so one of $\R$, $\C$\, $\H$, or $\O$). The multiplication structure is as follows.

For any nonzero $x\in V$ the space $A_x$ of \eqref{Ax} is the direct sum
$\mathbb Rx \oplus J_{\mathfrak n_{-2}}x=\mathbb R\oplus\mathfrak z$.
Using this identification, supply the space $A_x$ with the product
$$
(a,z)\cdot(a',z')=(a'',z''),\ \ \text{where}\ \ (a+J_z)(a'+J_{z'})x=(a''+J_{z''})x.
$$
This formula also makes sense for pseudo $H$-type algebras with the $J^2$ condition,
provided that $x$ is non-null (in which case we identify $A_x=\R\oplus\z$). 

The above product $\cdot$ is bilinear, but not necessary associative in general. 
If $\langle z, z\rangle_{\mathfrak z}\neq-a^2$, then the inverse element to $(a,z)$ is equal to
$$
(a,z)^{-1}=\Bigl(\frac{a}{a^2+\langle z, z\rangle_{\mathfrak z}},
\frac{-z}{a^2+\langle z, z\rangle_{\mathfrak z}}\Bigr).
$$
This shows, that for positive-definite metric $\langle\cdot,\cdot\rangle_{\mathfrak z}$,
$A_x=\mathbb R\oplus\mathfrak z$ is a division algebra, while in the
sign-indefinite og negative definite cases it is not.

This agrees with Theorem \ref{th:J2}, where $\dim\z\in\{0,1,3,7\}$. 
The case $\z=0$ is trivial, then $A_x=\R$.
In the other cases, when $\langle\cdot,\cdot\rangle_{\mathfrak z}$ is not positive-definite, we obtain
the split versions of the division algebras.

 \begin{proposition}
The algebra $A_x=\mathbb R\oplus\mathfrak z$ in the cases of pseudo $H$-type Lie algebras 
$\mathfrak n^{0,1}(V)$, $\mathfrak n^{1,2}(V)$, and $\mathfrak n^{3,4}(V)$ from Theorem \ref{th:J2}
is isomorphic to the algebra of split-complex $\C_s$, split-quaternion $\H_s$ or 
split-octonion numbers $\O_s$, respectively.
 \end{proposition}
 
 \begin{proof}
The minimal admissible module $A_x=\mathbb Rx \oplus J_{\mathfrak n_{-2}}x$ (with $x$ non-null)
from the pseudo $H$-type Lie algebra $\mathfrak n^{0,1}$ has 
the structure of split-complex numbers $\C_s$ due to the presence of a linear transformation 
$J_z$, $\langle z,z\rangle_{\mathfrak n_{-2}}=-1$, with $J^2_z={\rm Id}$. 

The minimal admissible module $A_x$ of $\mathfrak n^{1,2}$ has the structure of 
split-quaternion numbers $\H_s$, since the Clifford representations $J_{z_1},J_{z_2},J_{z_3}$, where $z_1,z_2,z_3$ are orthonormal basis for $\mathfrak n_{-2}$, satisfy relations~\eqref{eq:12}. 

The split-octonion structure $\O_s$ on the minimal admissible module $A_x$ of $\mathfrak n^{3,4}$ 
is formed by the Clifford representations $J_{z_j}$, $j=1,\ldots,7$ for an orthonormal basis for 
$\z=\mathfrak n_{-2}$. The split-octonionic relations among $J_{z_j}$ can be verified by using 
the involutions $P_1,P_2,P_3$ and $P_4=J_{z_5}J_{z_6}J_{z_7}$ from the proof of Theorem~\ref{th:J2}.
 \end{proof}

%%%%%%%%%%%%%%%%%%%%%%%%%

\section{Realization of $J^2$ condition via simple Lie algebras}\label{Ap.B}

Gradings of simple Lie algebras are enumerated by a choice of parabolic subalgebra, 
or in combinatorial terms by a choice of crossed nodes on the Dynkin (in the complex case) or 
the Satake (in the real case) diagrams.
Note that in the real case crosses can be put only on white nodes, and those connected by arrows shall
be crossed simultaneously, see \cite{OV,CS}. 

Real rank one simple Lie algebras have a unique grading (up to an internal automorphism),
but the other non-compact simple Lie algebras with the same complexification have higher ranks
and more choices of grading. There is however a unique choice in each case giving
a pseudo $H$-type algebra with the $J^2$-condition.
 % $\g=\g_{-\nu}\oplus\dots\oplus\g_0\oplus\dots\oplus\g_\nu$ will be called $|\nu|$-grading.

 \begin{theorem}\label{simple}
Every  pseudo $H$-type algebra satisfying the $J^2$-condition, is the negative graded part 
(the nilradical of the opposite parabolic) of one of the graded simple Lie algebras $\g$ 
from the following list, numerated as in Theorem ~\ref{th:J2}.
 \end{theorem}
 
%\bigskip
 
(1) $BD_\ell/P_1$

\vskip-5pt 
   
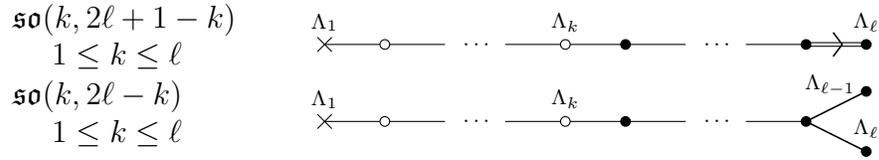
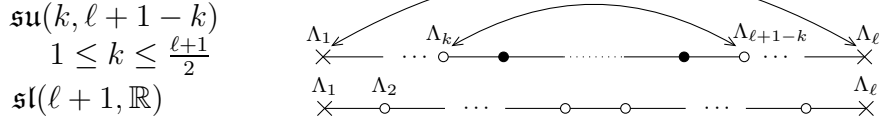
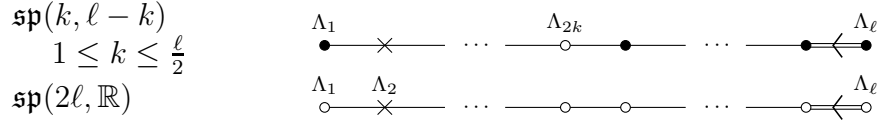
 \[
\parbox{3cm}{$\mathfrak{so}(k,2\ell+1-k)$\\
\hphantom{A}\, $1\le k\le\ell$}\qquad
{
 \begin{tiny}
 \begin{tikzpicture}[scale=0.8,baseline=-0pt]
\bond{0,0}; \bond{1,0}; \tdots{2,0}; \bond{3,0}; \bond{4,0}; \bond{5,0}; \tdots{6,0}; 
\bond{7,0}; \dbond{r}{8,0}; 
\DDnode{x}{0,0}{$\Lambda_1$}; \DDnode{w}{1,0}{}; \DDnode{w}{4,0}{$\Lambda_k$}; 
\DDnode{b}{5,0}{}; \DDnode{b}{8,0}{}; \DDnode{b}{9,0}{$\Lambda_{\ell}$}; 
 \useasboundingbox (-0.4,-0.2) rectangle (9.2,0.55); % make bounding box bigger
 \end{tikzpicture}
 \end{tiny}
 } 
 \]

\vskip-10pt 

  \[
\parbox{3cm}{$\mathfrak{so}(k,2\ell-k)$\\
\hphantom{A}\, $1\le k\le\ell$}\qquad
{
 \begin{tiny}
 \begin{tikzpicture}[scale=0.8,baseline=-0pt]
\bond{0,0}; \bond{1,0}; \tdots{2,0}; \bond{3,0}; \bond{4,0}; \bond{5,0}; \tdots{6,0}; 
\bond{7,0}; \diagbond{u}{8,0}; \diagbond{d}{8,0}; 
\DDnode{x}{0,0}{$\Lambda_1$}; \DDnode{w}{1,0}{}; \DDnode{w}{4,0}{$\Lambda_k$}; 
\DDnode{b}{5,0}{}; \DDnode{b}{8,0}{}; \DDnode{b}{9,0.5}{}; \DDnode{b}{9,-0.5}{$\Lambda_{\ell}$}; 
\node (mark) at (8.4,0.6) {$\Lambda_{\ell-1}$};
 \useasboundingbox (-0.4,-0.2) rectangle (9.2,0.55); % make bounding box bigger
 \end{tikzpicture}
 \end{tiny}
 } 
 \]
 
\medskip 

(2) $A_\ell/P_{1,\ell}$ 

\vskip-25pt 
 
 \[
\parbox{3cm}{$\mathfrak{su}(k,\ell+1-k)$\\
\hphantom{A}\, $1\le k\le\tfrac{\ell+1}2$}\qquad
{
 \begin{tiny}
 \begin{tikzpicture}[scale=0.8,baseline=4pt]
\bond{0,0}; \tdots{1,0}; \bond{2,0}; \bond{3,0}; \dotbond{4,0}; 
\bond{5,0}; \bond{6,0}; \tdots{7,0}; \bond{8,0}; 
\DDnode{x}{0,0}{$\Lambda_1$\ \ }; \DDnode{w}{2,0}{$\Lambda_k$\ \ }; 
\DDnode{b}{3,0}{}; \DDnode{b}{6,0}{}; \DDnode{w}{7,0}{$\qquad\ \Lambda_{\ell+1-k}$}; 
\DDnode{x}{9,0}{\ $\Lambda_\ell$};
\node (B) at (2,0.1) {}; \node (C) at (7,0.1) {}; \path[<->,font=\scriptsize,>=angle 45] (B) edge [bend left] (C);
\node (A) at (0,0.1) {}; \node (D) at (9,0.1) {}; \path[<->,font=\scriptsize,>=angle 45] (A) edge [bend left] (D);
 \useasboundingbox (-0.4,-0.2) rectangle (9.2,0.55); % make bounding box bigger
 \end{tikzpicture}
 \end{tiny}
 }
 \]
 
\vskip-10pt 
  
 \[
\parbox{3cm}{$\mathfrak{sl}(\ell+1,\mathbb{R})$
\hphantom{A}}\qquad
{
 \begin{tiny}
 \begin{tikzpicture}[scale=0.8,baseline=-0pt]
\bond{0,0}; \bond{1,0}; \tdots{2,0}; \bond{3,0}; \bond{4,0}; \bond{5,0}; \tdots{6,0}; 
\bond{7,0}; \bond{8,0}; 
\DDnode{x}{0,0}{$\Lambda_1$}; \DDnode{w}{1,0}{$\Lambda_2$}; \DDnode{w}{4,0}{}; 
\DDnode{w}{5,0}{}; \DDnode{w}{8,0}{}; \DDnode{x}{9,0}{$\Lambda_{\ell}$}; 
 \useasboundingbox (-0.4,-0.2) rectangle (9.2,0.55); % make bounding box bigger
 \end{tikzpicture}
 \end{tiny}
 } 
 \]  

\medskip 

(3) $C_\ell/P_2$ 

\vskip-5pt 
  
 \[
\parbox{3cm}{$\mathfrak{sp}(k,\ell-k)$\\
\hphantom{A}\, $1\le k\le\tfrac{\ell}2$}\qquad
{
 \begin{tiny}
 \begin{tikzpicture}[scale=0.8,baseline=-0pt]
\bond{0,0}; \bond{1,0}; \tdots{2,0}; \bond{3,0}; \bond{4,0}; \bond{5,0}; \tdots{6,0}; 
\bond{7,0}; \dbond{l}{8,0}; 
\DDnode{b}{0,0}{$\Lambda_1$}; \DDnode{x}{1,0}{}; \DDnode{w}{4,0}{$\Lambda_{2k}$}; 
\DDnode{b}{5,0}{}; \DDnode{b}{8,0}{}; \DDnode{b}{9,0}{$\Lambda_{\ell}$}; 
 \useasboundingbox (-0.4,-0.2) rectangle (9.2,0.55); % make bounding box bigger
 \end{tikzpicture}
 \end{tiny}
 } 
 \]  
  
\vskip-10pt 
  
 \[
\parbox{3cm}{$\mathfrak{sp}(2\ell,\mathbb{R})$
\hphantom{A}}\qquad
{
 \begin{tiny}
 \begin{tikzpicture}[scale=0.8,baseline=-0pt]
\bond{0,0}; \bond{1,0}; \tdots{2,0}; \bond{3,0}; \bond{4,0}; \bond{5,0}; \tdots{6,0}; 
\bond{7,0}; \dbond{l}{8,0}; 
\DDnode{w}{0,0}{$\Lambda_1$}; \DDnode{x}{1,0}{$\Lambda_2$}; \DDnode{w}{4,0}{}; 
\DDnode{w}{5,0}{}; \DDnode{w}{8,0}{}; \DDnode{w}{9,0}{$\Lambda_{\ell}$}; 
 \useasboundingbox (-0.4,-0.2) rectangle (9.2,0.55); % make bounding box bigger
 \end{tikzpicture}
 \end{tiny}
 } 
 \]  
 
\medskip 

(4) $F_4/P_4$ 

\vskip-8pt 
  
 \[
\parbox{1cm}{FII}\quad
{
 \begin{tiny}
 \begin{tikzpicture}[scale=0.8,baseline=-0pt]
\bond{0,0}; \dbond{r}{1,0}; \bond{2,0}; 
\DDnode{b}{0,0}{}; \DDnode{b}{1,0}{}; \DDnode{b}{2,0}{}; \DDnode{x}{3,0}{}; 
 \useasboundingbox (-0.4,-0.2) rectangle (3.2,0.55); % make bounding box bigger
 \end{tikzpicture}
 \end{tiny}
 } 
\hspace{75pt}\parbox{1cm}{FI}\quad
{
 \begin{tiny}
 \begin{tikzpicture}[scale=0.8,baseline=-0pt]
\bond{0,0}; \dbond{r}{1,0}; \bond{2,0}; 
\DDnode{w}{0,0}{}; \DDnode{w}{1,0}{}; \DDnode{w}{2,0}{}; \DDnode{x}{3,0}{}; 
 \useasboundingbox (-0.4,-0.2) rectangle (3.2,0.55); % make bounding box bigger
 \end{tikzpicture}
 \end{tiny}
 }  
 \]   
 
\medskip 
 
Note that the Satake diagrams for $C_\ell$ and $D_\ell$ cases shall be accordingly modified
when $k=\frac{\ell}2$ and $k=\ell-1,\ell$ respecively. The real rank of the Lie algebras is
$k$ when this is applicable, and in the other cases it is $\ell$, $\ell$, $1$ and $4$, respectively. 
The cases of real rank one conicide with the $H$-type Lie algebras studied in \cite{CDKR}.

 \begin{proof}
In case (1), $\n=\g_{-1}$ is simply a vector space with an inner product, and the 
$M$-type algebra is given by the signature of this inner product
$\langle\cdot,\cdot\rangle_{\n_{-1}}$, which is encoded by $k$ and $\dim\n$ 
(or $\ell$ and the choice of diagram $B$ or $D$).

In case (2), $\n=\g_{-2}\oplus\g_{-1}$ is the Heisenberg algebra. 
When $\langle\cdot,\cdot\rangle_{\n_{-2}}$ is positive, the algebra $\n^{1,0}(V)$ is given by 
$\C$-module $V$, which is the sum of $\ell=\frac12\dim\n_{-1}$ irreducible modules $\C$,
and then $\langle\cdot,\cdot\rangle_{\n_{-1}}$ is uniquely specified by the signature that is
encoded by $k$ in $\mathfrak{su}(k,\ell+1-k)$. The algebra $\n^{0,1}(V)$ is unique,
since $V$ is obtained by doubling and the signature is split. This corresponds to 
$\mathfrak{sl}(\ell+1,\R)$.
 
In case (3), the series $\mathfrak{sp}(k,\ell-k)$ is similar to the rank $k=1$ algebra:
all cases have type $\n^{3,0}(V)$ and differ by the signature of the metric on $V$,
which is encoded by $k$. The algebra $\mathfrak{sp}(2\ell,\R)$ corresponds to $\n^{1,2}(V)$,
with the unique module $V$ that is obtained by doubling and has a split signature metric.
Let us give details of this latter identification.
 
Recall that $\mathfrak{sp}(2\ell,\mathbb{R})=\left\{
\begin{pmatrix}D & B\\ C & -D^t\end{pmatrix}:B,C,D\in\mathfrak{gl}_{\ell}(\R),B^t=B,C^t=C\right\}$.
The grading corresponding to the parabolic $\mathfrak{p}_2$ has the following negative part:
 $$
A=\begin{bmatrix}0 & 0 & 0 & 0\\ Y & 0 & 0 & 0\\ Z & X^t & 0 & -Y^t\\ X & 0 & 0 & 0\end{bmatrix}\qquad
\parbox{7cm}{$\g_{-1}=\R^{4(\ell-2)}(X,Y)$\\ \hphantom{AAA} $X,Y$ are $(\ell-2)\times2$ matrices\\
$\g_{-2}=\R^3(Z)$\\ \hphantom{AAA} $Z^t=Z$ is a $2\times2$ matrix}
  $$

If we write $Z=\begin{bmatrix}z_{11} & z_{12}\\ z_{12} & z_{22}\end{bmatrix}$,
$X=\begin{bmatrix}X_1\\ X_2\end{bmatrix}$, $Y=\begin{bmatrix}Y_1\\ Y_2\end{bmatrix}$,
where $X_1,X_2,Y_1,Y_2\in\R^{\ell-2}$ are row-vectors, then the 
brackets of  $\n=\g_{-}=\g_{-2}\oplus\g_{-1}$ are given by
$[A',A'']=A$, where
$z_{ij}=\langle X'_i,Y''_j\rangle+\langle X'_j,Y''_i\rangle-\langle X''_i,Y'_j\rangle-\langle X''_j,Y'_i\rangle$
with respect to the standard scalar Euclidean product $\langle\cdot,\cdot\rangle$ on $\R^{\ell-2}$
(we use the same notations for coordinates of $A',A'',A$ marking them with prime, double prime or nothing).

This means that $\n_{-1}=\oplus_{i=1}^{\ell-2}\mathbb{S}_i$ is the sum of standard (doubled) 4D modules
over the Clifford algebra of $\z=\n_{-2}$ and the brackets are $[\mathbb{S}_i,\mathbb{S}_j]=0$
for $i\neq j$, as well as the same nontrivial map $\Lambda^2\mathbb{S}_i\to\z$ for all $i$.

Let us describe these latter brackets for one module $\mathbb{S}$. 
Denote by $e_1,e_2,f_1,f_2\in\mathbb{S}$, $h_{11},h_{12},h_{22}\in\z$ the bases corresponding 
to the above coordinates. 
The only non-trivial brackets are $[e_1,f_1]=2h_{11}$, $[e_1,f_2]=[e_2,f_1]=h_{12}$, $[e_2,f_2]=2h_{22}$.

The metrics of $\n_{-1}$ and $\n_{-2}$ are given by 
$\langle e_2,f_1\rangle_{\n_{-1}}=-\langle e_1,f_2\rangle_{\n_{-1}}=1$ and
$\langle h_{11},h_{22}\rangle_{\n_{-2}}=-\frac12$, $\langle h_{12},h_{12}\rangle_{\n_{-2}}=1$
(these are the only non-trivial scalar products),
whence the representaiton $J:\z\to\op{End}(\mathbb{S})$:
 $$
J_{h_{11}}= \begin{bmatrix}0 & 1 & 0 & 0\\ 0 & 0 & 0 & 0\\ 0 & 0 & 0 & 1\\ 0 & 0 & 0 & 0\end{bmatrix},\
J_{h_{12}}= \begin{bmatrix}-1 & 0 & 0 & 0\\ 0 & 1 & 0 & 0\\ 0 & 0 & -1 & 0\\ 0 & 0 & 0 & 1\end{bmatrix},\
J_{h_{22}}= \begin{bmatrix}0 & 0 & 0 & 0\\ -1 & 0 & 0 & 0\\ 0 & 0 & 0 & 0\\ 0 & 0 & -1 & 0\end{bmatrix}.
  $$
It is straightforward now to check the pseudo $H$-type and the $J^2$-conditions. 

We can explain this more representation-theoretically. Let $V$ be the standard $\mathfrak{sl}_2$ module.
The module $\z=\mathfrak{ad}=S^2V$ has a unique up to constant factor invariant metric 
$\langle\cdot,\cdot\rangle_{\n_{-2}}$ of signature $(1,2)$, 
which is indeed the Killing form of $\mathfrak{sl}_2$.
The module $\mathbb{S}=V\oplus V^*\simeq V\oplus V$ has also unique up to constant factor
invariant metric coming from the pairing between $V$ and $V^*$, 
and it is $\langle\cdot,\cdot\rangle_{\n_{-1}}$ (this doubling of $V$ is in agreement with Table \ref{t:dim}).

With these choices the universal envelopping algebra $U(\mathfrak{sl}_2)$ action on $\mathbb{S}$
agrees with the Clifford relations, and it is easy to see that this $M$-type algebra is isomorphic to the one 
coming from the Lie algebra $\mathfrak{sp}(6,\R)$ above, which corresponds  to the partial case 
$\ell=3$ of the split form of $C_\ell$ with one irreducible module $\mathbb{S}$.

Finally, in case (4) we have only two algebras. The algebra $\n^{7,0}(V)$ 
corresponds to the real rank one case FII considered in \cite{CDKR}, and 
completely analogously the algebra $\n^{3,4}(V)$ corresponds to FI.
 \end{proof}

Note that in both cases (1) and (2) the Lie algebra $\n$ is of infinite type:
it is the algebra of all formal vector fields on $\n_-$ in the $|1|$-graded case 
and the algebra of formal contact vector fields in the case of contact grading.

On the contrary, cases (3) and (4) are rigid: the Tanaka prolongation of $\n=\g_-$ coincides 
with $\g$. This follows from Yamaguchi's theorem \cite{Y} and agrees with the classification 
\cite{AS2}. Thus these cases are of finite type.

 \begin{remark}
{\rm 
The Lie algebras of all cases in (1) have the same Abelian structure for all $k$, 
it is the metric structure that varies. Similarly, all cases in (2) correspond to the same
Heisenberg algebra (the theorem makes a distinction of $M$-type).

On the other hand, all cases in (3) and (4) are pairwise different as Lie algebras.
Indeed, by the observation before this remark the Tanaka prolongations of these 2-step nilpotent
Lie algebras coincide with simple Lie algebras $\g$ specified in Theorem \ref{simple}.
Since these latter are different, so are the nilpotent Lie algebras $\n$.}
 \end{remark}

This finishes realization and classification of all pseudo $H$-type algebras satisfying the $J^2$ 
condition via nilradicals of parabolics in simple Lie algebras.
 
%%%%%%%%%%%%%%%%%%%%%%%%%

\end{document}